\documentclass[10pt,twoside,english,reqno,a4paper]{amsproc}

\usepackage{geometry}
\usepackage{amsmath}
\newlength{\defbaselineskip}
\usepackage{xcolor}

\usepackage{listings,graphicx,varioref,color,bm,stmaryrd}

\usepackage{soul,xcolor}

\setstcolor{red}

\usepackage{calrsfs}

\usepackage{esint}

\usepackage{appendix}

\usepackage{mathrsfs}  
\usepackage{makecell}
\usepackage[mathscr]{euscript}

\usepackage{mathpazo} 
\usepackage[scaled=.95]{helvet} 

\renewcommand{\iint}{{\ds\iint}}

\usepackage{amsthm}
\usepackage{mathpazo} 
\usepackage[scaled=.95]{helvet} 
\usepackage{courier}
\usepackage{amssymb}
\usepackage{amsmath}
\usepackage{amscd}
\usepackage{amsfonts}
\usepackage[T1]{fontenc}
\usepackage{graphics}
\usepackage[english]{babel}
\usepackage[utf8]{inputenc} 
\usepackage[makeroom]{cancel}
\usepackage{pifont}
\usepackage{marvosym}
\usepackage{wasysym}
\usepackage{makecell}
\usepackage{tabu}
\usepackage{multirow}
\usepackage{multicol}

\usepackage{enumerate}
\usepackage{graphics}
\usepackage{pgfplots}
\pgfplotsset{width=7cm,compat=newest} 
\usepackage{graphics}
\usepackage{tikz}
\usetikzlibrary{shapes.geometric}
\usepackage{tikz-cd}
\usepackage[labelsep=endash]{caption}
\usepackage{float}

\theoremstyle{plain}
\newtheorem{theorem}{Theorem}[section]
\newtheorem{theo}{Theorem}

\newtheorem{lemma}[theorem]{Lemma}
\newtheorem{corollary}[theorem]{Corollary}

\theoremstyle{definition}
\newtheorem{defin}[theorem]{Definition}

\newtheorem{remark}[theorem]{Remark}
\theoremstyle{remark}

\renewcommand{\theequation}{\thesection.\arabic{equation}}
\numberwithin{equation}{section}
\long\def\salta#1{\relax}

\definecolor{bor}{cmyk}{0.21,0.93,0.86,0.12}
\definecolor{air}{rgb}{0.178, 0.51, 0.51}
\definecolor{range}{cmyk}{0,0.599,1,0.188}

\usepackage[normalem]{ulem}               

\makeatletter

\def\og{\leavevmode\raise.3ex\hbox{$\scriptscriptstyle\langle\!\langle$~}}
\def\fg{\leavevmode\raise.3ex\hbox{~$\!\scriptscriptstyle\,\rangle\!\rangle$}}

\usepackage{hyperref}
\hypersetup{colorlinks=true,linkcolor=red!70!black,citecolor=orange,urlcolor=blue}

\usepackage{xmpmulti}
\usepackage[justification=centering,labelformat=empty]{caption}

\usepackage[pages=some,scale=1,angle=0,opacity=0.7]{background}

\usepackage{graphicx}
\usepackage{graphics}
\usepackage{pgfplots}

\pgfplotsset{width=7cm,compat=newest}

\usepackage{tikz,lmodern}
\tikzstyle{na} = [baseline=-.5ex,remember picture]

\usetikzlibrary{tikzmark,fit,shapes.geometric}
\usepackage{tikz-cd}
\usetikzlibrary{shapes.geometric,calc,shapes,matrix,arrows,fit,decorations.text}

\usepackage{tkz-graph}

\usepackage{caption}
\usepackage{subcaption}

\usepackage{float}

\usepackage[normalem]{ulem}               
\newcommand\redout{\bgroup\markoverwith
{\textcolor{red}{\rule[0.5ex]{2pt}{0.8pt}}}\ULon}

\usetikzlibrary{patterns}
\usepackage{pgfplotstable}
\usetikzlibrary{intersections}

\usetikzlibrary{patterns,hobby}
        \usepackage{pgfplots}
        \pgfplotsset{compat=1.6}

\def\ok{\bar{k}}

\def\q{\quad}
\def\qq{\qquad}
\def\gz{{\gamma_0}}
\def\gi{{\gamma_\infty}}

\def\de{\delta}
\def\eps{\varepsilon}

\newcommand{\bc}{\overline{c}}
\newcommand{\uc}{\underline{c}}

\def\D{\Delta}
\def\vp{\varphi}
\def\be{\begin{equation}}
\def\ee{\end{equation}}

\def\al{\alpha}
\def\de{\delta}
\def\vp{\varphi}
\def\lm{\lambda}
\def\si{\sigma}
\def\vare{\varepsilon}

\def\ti{{\theta_\infty}}
\def\tz{\theta_0}
\def\wn{w_n}
\def\un{u_n}
\def\b{\beta}
\newcommand{\N}{\nabla}

\newcommand{\NN}{\mathbb{N}}
\newcommand{\intO}{\int_{\Omega}}
\newcommand{\ds}{\displaystyle}

\DeclareMathOperator{\R}{\mathbb{R}}

\def\om{\omega}
\def\uom{{s_1}}
\def\oom{{s_2}}
\def\cu{{\underline{c}}}
\def\co{\overline{c}}

\newcommand{\pare}[1]{\left( #1 \right)}
\newcommand{\norm}[1]{\left\| #1 \right\|}

\newcommand{\bra}[1]{\left[ #1 \right]}
\renewcommand{\t}[1]{\text{#1}}
\newcommand{\set}[1]{\left\{ #1 \right\}}

\newcommand{\Tk}[1]{T_k\pare{#1}}
\newcommand{\Gk}[1]{G_k\pare{#1}}

\def\XXint#1#2#3{{\setbox0=\hbox{$#1{#2#3}{\int}$}
  \vcenter{\hbox{$#2#3$}}\kern-.5\wd0}}

\usepackage{datetime}
\usepackage{fancyhdr}
\author[M. Magliocca]{Martina Magliocca}
\author[F. Oliva]{Francescantonio Oliva}

\address[M. Magliocca]{Centre Borelli, ENS Paris-Saclay, 4 Avenue des Sciences, 91190 Gif-sur-Yvette, France 
	\\ \url{mmaglioc@ens-paris-saclay.fr}}

\address[F. Oliva]{Dipartimento di Matematica e Applicazioni, Universit\`a di Napoli Federico II, Via Cintia, Monte S. Angelo, 80126 Napoli, Italy 
\\ \url{francescantonio.oliva@unina.it}}

\keywords{Nonlinear parabolic equations, Singular parabolic equations, Repulsive Gradient} \subjclass[2000]{35J60, 35J61, 35J75, 35R06}

\begin{document}

\title{{ On some parabolic  equations involving superlinear singular gradient terms }}

\begin{abstract}
In this paper we prove existence of   nonnegative solutions to parabolic Cauchy-Dirichlet problems with superlinear gradient terms which are possibly singular. The model equation is
\[
 u_t - \D_pu=g(u)|\N u|^q+h(u)f(t,x)\qq \t{in }(0,T)\times\Omega,
\]
where $\Omega$ is an open bounded subset of $\R^N$ with $N>2$, $0<T<+\infty$, $1<p<N$,  and $q<p$ is superlinear. 
The functions $g,\,h$ are continuous and possibly satisfying $g(0) = +\infty$ and/or $h(0)= +\infty$,
with different rates. Finally, $f$ is nonnegative and it belongs to a suitable Lebesgue space. We investigate the relation among the superlinear threshold of $q$, the regularity of the initial datum and the forcing term, and the decay rates of $g,\,h$ at infinity.

\end{abstract}

\setcounter{tocdepth}{4}
\setcounter{secnumdepth}{4}
\maketitle
\tableofcontents

\section{Introduction}

We are interested in existence of  solutions for   Cauchy-Dirichlet parabolic problems with possibly singular nonlinear first order terms. The model problem we have in mind is  presented below:
\begin{equation}
\begin{cases}
\displaystyle u_t - \D_pu= g(u)|\N u|^q+ h(u)f(t,x) &  \text{in}\, Q_T, \\
\displaystyle u\ge 0 &  \text{in}\, Q_T, \\
 u=0 & \text{on}\ (0,T) \times \partial \Omega,\\
 u(0,x)=u_0(x) & \text{in}\, \Omega,
\label{p}
\end{cases}
\end{equation} 
where $\Delta_p u= \operatorname{div(|\nabla u|^{p-2}\nabla u)}$ with $1<p<N$, $q<p$, $Q_T = (0,T)\times\Omega$ is the parabolic cylinder where $0<T<+\infty$ and $\Omega$ is a bounded open subset of $\mathbb{R}^N$, $N>2$. The nonnegative functions $u_0,\,f$ belong to suitable Lebesgue spaces, while $g(s), h(s)$ are continuous nonnegative functions which are admitted to satisfy $g(0) = +\infty$ and/or $h(0)= +\infty$, finite elsewhere and possibly decaying to zero as $s\to+\infty$. Therefore, problem \eqref{p} is \textit{singular} since we are asking the solution to be zero on $(0,T) \times \partial \Omega$ and at the same time the right-hand tends to explode. Additionally, problem \eqref{p} is \textit{superlinear} since, as we will see, the growth of the first order term requires compatibility conditions among the data.

\medskip

When $g,h=1$ physical motivations for the study of \eqref{p} are related to the approximation, in a viscous sense, of Hamilton-Jacobi equations (see \cite{l2}); when $q=p=2$, they are also known as Kardar–Parisi–Zhang equations and they are connected to the theory of growth and roughening of surfaces (see \cite{kpz}).\\
On the other side, the first time that the singular case $g\equiv 0, h(s)=s^{-1}$ appeared was in \cite{FM}; here the authors fall into the study of problems as in \eqref{p} while observing the temperature (given by the solution $u(t,x)$) of an electrical conductor which occupies a three dimensional region. Here $f(t,x)u^{-1}$ is thought as the rate of generation of heat, where $h(s)=s^{-1}$ is the resistivity of the conductor.

\medskip

From a purely mathematical point of view the literature is wide. Firstly let us underline that we will regularly refer to the "superlinear" character of the first order term which is fairly explained in Section \ref{appendix} below. Here, at first glance, one can think of a power of the gradient that, in some sense, grows more than the principal operator on the left-hand. To give an idea, if we look at the homogeneity of the equation
\[
u_t-\D u=|\N u|^q\q\t{with } q<2,
\]
then we have a superlinear growth in the gradient if $q>1$. This will be translated into the necessity of imposing suitable compatibility conditions among the unbounded data which are discussed into Sections \ref{sec:non} and \ref{appendix}. 
\\
In order to understand our natural superlinear setting, we  focus on the stationary equation
\[
-\D_pu=|\N u|^q+f\q\t{in } \Omega.
\]
Its general version has been deeply investigated in \cite{GMP}, when the parameter $q$ satisfies $p-1<q<p$, and we have a superlinear gradient growth;  an optimal Lebesgue space is identified 
in which $f$ needs to be taken in order to have the existence of a suitable notion of solution. 
\\ 

A similar fact  has been shown for the parabolic case in \cite{BASW}; here the authors prove the existence of solutions to
\[
\begin{cases}
u_t-\Delta u=\lambda |\N u|^q&\t{in } (0,T)\times\mathbb{R}^N,\,\lm\in\R\setminus\set{0},\\
u(0,x)=u_0(x)&\t{in } \mathbb{R}^N,
\end{cases}
\]
through semigroup theory and heat kernel estimates.
Moreover, as already pointed our for the elliptic case, a suitable compatibility condition is needed.
Roughly speaking, solutions are admitted if $u_0$ belongs at least to a suitable Lebesgue space $L^\si(\Omega)$ for a given value $\si=\si(q)$; otherwise nonexistence of solutions is shown.
We underline that techniques employed in \cite{BASW} are strongly related to the linearity of the operator and the results can not immediately be extended to the nonlinear case. We also refer the interested reader to \cite{An} where the author shows existence and nonexistence results for a more general equation in a slightly different framework and with a different technique.
\\

In addition, if we look at the nonsingular superlinear version of \eqref{p}, namely $g,h\equiv 1$, it is clear that we cannot simply assume $q>p-1$.
In \cite{M} (see also \cite{DNFG}) an existence result is proven for rough data both in the superlinear and in the sublinear setting. Here it is observed that the presence of the time derivative strongly influences the relation between the growth rate $q$ and $p$. Indeed, when $p\ne 2$, then the operator $u_t-\D_pu$ is not homogeneous. This fact implies that we have two different superlinear thresholds, one for the case $p>2$ and the other when $p<2$.\\
Moreover, as already pointed out, one has to require initial data which satisfy a well determined compatibility condition.   Furthermore, in this nonhomogeneous case, the forcing term $f$ has to verify a regularity assumption as well. \\
We finally highlight that this issue is independent of the nature of the superlinearity (i.e. zero or first order). We quote, in this sense, the works of \cite{BC,Padv} for parabolic problems with superlinear zero lower order terms,  \cite{SMM} in the stationary setting and with superlinearities of the type $g(u)|\N u|^q$ with a bounded $g$ decaying at infinity.

\medskip

Even the singular case of problem \eqref{p}, namely or $g(0)=+\infty$ and/or $h(0)=+\infty$, has been widely studied. 
\\For the stationary setting, when $g\equiv 0$ and $h(s)=s^{-\gamma}$ ($\gamma>0$), the masterpiece works are \cite{CRT,LM} where the authors prove existence of a classical solution. More recently, in presence of $L^1$ or measure data $f$ and for a more general $h$, we quote \cite{BO,ddo,op0}. Here, the problem is mainly dealt through an approximation scheme. Two features need to be highlighted: firstly, roughly speaking, when the growth in zero is too strong then the solution will have just local finite energy and the classical boundary trace will be lost. Secondly, when passing to the limit, a particular care is clearly given to the zone of degeneracy of the solution, showing  that some positiveness is necessary in order to take the limit.\\    
When $g(0)=+\infty$ and $q=p$, problem \eqref{p} has been already analyzed in \cite{ABLP,gps1,gps2} where the authors
prove existence of a solution to the problem, under different assumptions on the growth condition on $g$ and for $h\equiv 1$. Clearly, 
even in this case, suitable compatibility conditions on $f$ are unavoidable to prove existence of the solution.\\

As one can expect, the literature concerning the evolutive counterpart of the singular problem is more limited. When $g\equiv 0$, $p\ge 2$ and $h(s)=s^{-\gamma}$ ($\gamma>0$), problem \eqref{p} is treated in \cite{dbdc}. Here, the authors prove existence of a distributional solution via an approximation argument and one of the main tools is a suitable application of the Harnack inequality in order to deduce the positivity of the approximating sequence. More recently, in presence of a general $h$ and measure data, existence and uniqueness has been addressed in \cite{op}, under suitable assumptions. 
\\
When $g(0)=+\infty$, $h\equiv 1$ and for some $1<q<p$, problem \eqref{p} has been already analyzed in \cite{DG} for bounded initial data $u_0$ and for $f$ satisfying the Aronson-Serrin curve. Under these assumptions, the authors prove the existence of a bounded solution. Moreover, as their proofs deeply exploit the inequality $|\N u|^q\le c|\N u|^p+c$, then their results can even be applied to problems with natural growth in the gradient, i.e. $q=p$. 
Furthermore, such an application of  Young's inequality implies that one needs suitable exponential test functions in order to approach to this critical growth. For further comments in this sense, we also refer to \cite{DGP,DGS,DGS2} where nonsingular parabolic problems with natural growth are considered. We also highlight that problem \eqref{p} has been studied in \cite{DG2} when $h\equiv1$ and for changing sign  $f$ and $u_0\in L^\infty(\Omega)$.

\medskip

In the current work the aim is twofold. Firstly we surely want to take advantage of the rate $q<p$ in order to get sharp existence results even for unbounded data. We want to understand the superlinear setting which is widely affected by the nonlinearity $g$, namely the rate at which $g$ degenerates at  infinity; as we will see, the faster $g$ is decaying at infinity, the lower is the regularity needed for initial data in order to have an existence result. An analogous phenomenon involves $f$ and $h$; 
the faster $h$ is decaying at infinity, the lower is the regularity needed on $f$ in order to have a solution, possibly requiring just $L^1$-data. 
We also mention that, coherently with the framework widely  described above, for a solution one has to think to a distributional solution having a suitable power  with finite energy, i.e. belonging to $L^p(0,T; W^{1,p}_0(\Omega))$. As we will see, there will be cases in which the behavior at zero of $g$, and $h$ could not permit to deduce this power property for the solution which will be given for its truncation from below. Let us highlight that the superlinear thresholds we consider are natural, as shown in Sections  \ref{sec:non} and \ref{appendix}, and that we recover the setting given in \cite{M} when $g,h\equiv 1$.\\
On the other hand, we want to investigate the superlinear framework for very general nonlinearities $g,h$ which possibly blow up at zero. This gives a significant role to the zone of degeneracy of the solution and it means that a particular attention is necessary in the passage to the limit. Moreover, when the behavior of either $g$ or $h$ at zero is too strong, the solution will attain the Dirichlet datum in a weaker sense and it will create  additional difficulties.

\medskip

The plan of the paper is the following. In Section \ref{sec:ass} we set the problem and state the main result. In Section \ref{secmain} we prove the existence of weak solutions in the mild singular case. In Section \ref{sec:strong} we treat the case where $g$ and/or $h$ are strongly singular at zero. In Section \ref{sec:non} we show that our assumptions are sharp and nonexistence of solutions is possible, at least when we have finite energy solutions. Section \ref{appendix} is devoted to understanding the superlinear threshold and the natural compatibility conditions on the data.

\subsection{Notations and preliminaries}

For a given function $s$ we denote its positive part $s^+=\max(s,0)$. For a fixed $k>0$, we define the truncation function $T_{k}:\R^+\to\R^+$ and the level set function $G_{k}:\R^+\to\R^+$  as follows\\
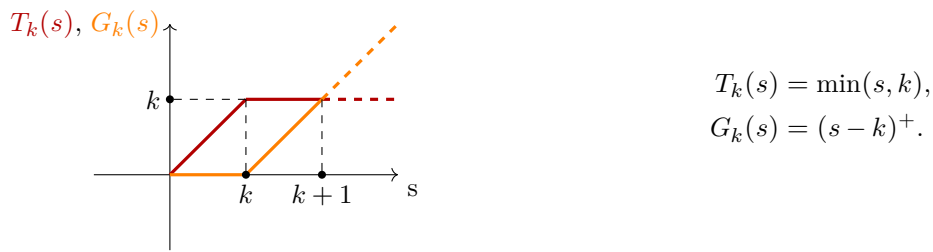
\begin{tabular}{ccc}
\begin{minipage}{.6\textwidth}
\begin{figure}[H]
\centering
\begin{tikzpicture}
\draw[->] (0,0) -- (3,0) node[anchor=north west] {s};
\draw[->] (0,0) -- (0,2) node[left] {{\color{red!70!black}$T_k(s)$}, {\color{orange}$G_k(s)$}};
\draw[-] (-1,0) -- (0,0);
\draw[-] (0,-1) -- (0,0);
\draw[very thick, red!70!black] (0,0) -- (1,1);
\draw[very thick, red!70!black] (1,1) -- (2,1);

\draw[very thick, dashed, red!70!black] (2,1) -- (3,1);

\draw[very thick, orange] (0,0) -- (1,0);
\draw[very thick, orange] (1,0) -- (2,1);

\draw[very thick, dashed, orange] (2,1) -- (3,2);

\draw[dashed] (1,1) -- (1,0);
 
\draw[dashed] (2,1) -- (2,0);
 
\draw[dashed] (1,1) -- (0,1);
 
\fill (1,0) circle (1.5pt) node[below] {$k$};
\fill (2,0) circle (1.5pt) node[below] {$k+1$};

\fill (0,1) circle (1.5pt) node[left] {$k$};
 
\end{tikzpicture}

\end{figure}
\end{minipage}
&
\begin{minipage}{.3\textwidth}
{\[\!
\begin{split}
T_k(s)&=\min (s,k),\\
G_k(s)&=(s-k)^+.
\end{split}
\]}
\end{minipage}
&
\begin{minipage}{.1\textwidth}
\end{minipage}
\end{tabular}
\\

\medskip
Note that  the following decomposition holds
\[
s=\Tk{s}+\Gk{s}.
\]

\noindent We denote by $\langle \cdot,\cdot\rangle$, in a standard way, the duality product between $W^{-1,p'}(\Omega)$ and $W^{1,p}_0(\Omega)$.
We recall that the Sobolev embedding exponent is defined as $ p^*=\frac{Np}{N-p}$, and the conjugate of $p$ is $p'=\frac{p}{p-1}$. 
We explicitly remark that, if not otherwise specified, we will denote by $c,M$ several positive constants whose values may change from line to line and, sometimes, on the same line. These constants will only depend on the data and parameters but they will never depend on the indices of the sequences we will introduce.\\

We will often make use of the following well known result.

\begin{theo}[Gagliardo-Nirenberg inequality]\label{teoGN}
Let $\Omega\subset \mathbb{R}^N, \,N\ge2$, be a bounded and open subset and let $T>0$. Then, if
\begin{equation*} 
v\in L^{\infty}(0,T;L^{h}(\Omega))\cap L^{\eta}(0,T;W_0^{1,\eta}(\Omega)),\q 1\le \eta<N \q \t{and}\q 1\le h\le\eta^*,
\end{equation*}
one has
\begin{equation*}
v\in L^a(0,T;L^b(\Omega)),
\end{equation*}
where the couple $(b,a)$ fulfills
\[
h\le b\le \eta^*,\q\eta\le a\le \infty,
\]
and satisfies the relation
\begin{equation}\label{rel}
\frac{Nh}{b}+\frac{ N(\eta-h)+\eta h }{a}=N.
\end{equation}
Moreover, the following inequality holds:
\begin{equation}\label{disGN}
\int_0^T \|v(t)\|_{L^b(\Omega)}^a \le c(N,\eta,h)\|v\|_{L^{\infty}(0,T;L^h(\Omega))}^{a-\eta}\int_0^T\|\N v(t)\|_{L^{\eta}(\Omega)}^{\eta} .
\end{equation}
In particular, having $a=b$ implies that  
\begin{equation*}
v\in L^{\eta\frac{N+h}{N}}(Q_T),
\end{equation*}
and the estimate reads 
\begin{equation}\label{disGN=}
\int_0^T \|v(t)\|_{L^{\eta\frac{N+h}{N}}(\Omega)}^{\eta\frac{N+h}{N}} \le c(N,\eta,h)\|v\|_{L^{\infty}(0,T;L^h(\Omega))}^{\frac{\eta h}{N}}\int_0^T\|\N v(t)\|_{L^{\eta}(\Omega)}^{\eta} .
\end{equation}
\end{theo}
\medskip

\section{Assumptions and main result in the superlinear and mild singular case}
\label{sec:ass}

Let us consider the following Cauchy-Dirichlet parabolic problem 
\begin{equation}
\begin{cases}
\displaystyle u_t - \t{div }a(t,x,u,\N u)=H(t,x,u,\N u)  &  \t{in}\, Q_T, \\
\displaystyle u\ge 0 &  \text{in}\, Q_T, \\
 u=0 & \t{on}\ (0,T) \times \partial \Omega,\\
 u(0,x)=u_0(x) & \t{in}\, \Omega,
\label{pb}
\tag{P}	
\end{cases}
\end{equation} 
where $\Omega$ is a bounded open subset of $\mathbb{R}^N$, $N>2$, $Q_T= (0,T)\times\Omega$ is the parabolic cylinder with $0<T<\infty$.

\medskip

Let $a:(0,T)\times\Omega\times\R\times \mathbb{R}^N\to\mathbb{R}^{N}$ be a Carath\'eodory function such that, for almost every $(t,x)\in Q_T$ and for all $(s,\xi)\in \mathbb{R}\times\mathbb{R}^N$, there exists $\alpha, c>0$ and $\ell\in L^{p'}(Q_T)$ providing: 
\begin{subequations}
\makeatletter
\def\@currentlabel{A}
\makeatother
\label{A}
\renewcommand{\theequation}{A\arabic{equation}}
\begin{align}
\label{A1}
& \alpha|\xi|^p\le a(t,x,s,\xi)\cdot\xi, \\
\label{A2}
&| a(t,x,s,\xi)|\le c\pare{s^{p-1}+|\xi|^{p-1}+\ell(t,x)},\\
\label{A3}
&\pare{a(t,x,s,\xi)-a(t,x,s,\eta)}\cdot(\xi-\eta)>0  \q \t{for all}\q  \xi\ne\eta,
\end{align}
\end{subequations}
where $1<p<N$ is great enough as pointed out below.

The function $H:(0,T)\times \Omega\times [0,+\infty)\times \mathbb{R}^N\to [0,+\infty]$   is a Carathéodory function as well, which satisfies the  following growth assumption: 
\begin{equation}\tag{H}\label{H}
\ds
H(t,x,s,\xi)\le   g(s)|\xi|^q + h(s)f
\end{equation}
for almost everywhere $(t,x)\in Q_T$, for all $s\in[0,+\infty)$, and for all $ \xi\in\mathbb{R}^N$.\\
Concerning $g,h$ we suppose that there exist $0< \uom\le \oom$ and $\cu,\,\co>0$ such that the function $g:[0,+\infty)\to [0,+\infty]$ is a continuous function which is finite away from the origin and such that
\begin{align}\tag{$\t{g}_0$}\label{g1}
&\displaystyle \exists \theta_{0}\ge 0:\q g(s)\le \frac{\uc}{s^{\theta_{0}}} \q \t{if} \ \ s<\uom, \\
\label{g2} \tag{$\t{g}_\infty$}
&\displaystyle \exists \theta_{\infty}\ge 0:\q g(s)\le \frac{\bc}{s^\ti} \q \t{if} \ \ s>\oom.
\end{align}
 In the same spirit $h:[0,+\infty)\to [0,+\infty]$ is a continuous function which is finite away from the origin and such that
\begin{align}\tag{$\t{h}_0$}\label{h1}
&\displaystyle \exists \gz\ge 0 :\q h(s)\le \frac{\uc}{s^\gz} \q \t{if} \ \ s<\uom, \\
\label{h2} \tag{$\t{h}_\infty$}
&\displaystyle \exists \gi\ge 0:\q h(s)\le \frac{\bc}{s^\gi} \q \t{if} \ \ s>\oom.
\end{align}
We deal with superlinear growths in $q$ which are subnatural, i.e. 
\begin{equation*}\label{Q}\tag{$\t{Q}$}
\max\set{
		\frac{p(1+\ti)}{2},\frac{N(p-1+\ti)+p(1+\ti)}{N+2}
	}<q<p.
\end{equation*}
As explained in Section \ref{appendix}, the needed Lebesgue regularity  on the initial datum is given by
\begin{equation}\tag{$\t{ID}$}\label{ID1}
0\le u_0\in L^{\sigma}(\Omega)\q\t{with}\q \si = \frac{N(q-p+1-\ti)}{p-q}.
\end{equation}
It is natural requiring $\si>1$ in \eqref{ID1} which takes $q$ to be in the following range 
\begin{align} \max\left\{\frac{p(1+\theta_{\infty})}{2},p-\frac{N(1-\ti)}{N+1}\right\}<q<p\label{Q1}\tag{$\t{Q}_1$}.
\end{align}
The remaining part of the superlinear range
\begin{align}\tag{$\t{Q}_2$}
& \max\left\{\frac{p(1+\theta_{\infty})}{2},\frac{N(p-1+\theta_{\infty})+p(1+\theta_{\infty})}{N+2}\right\}<q\le p-\frac{N(1-\ti)}{N+1}\q \t{for}\q p>\frac{2N}{N+1}\label{Q2}
\end{align}
can be dealt with $0\le u_0\in L^\si(\Omega)$ for any $\si\in(1,2)$. Here, the bound from below on $p$ is needed in order to have $\frac{p(1+\theta_{\infty})}{2}<p-\frac{N(1-\ti)}{N+1}$.

The nonnegative forcing term $f:(0,T)\times \Omega\to[0,+\infty)$ has to satisfy its own compatibility condition with respect to the superlinearity. These restrictions on the regularity of $f$ can be removed when $\gi$ is large enough.\\
We set $f$ in $ L^{r}(0,T;L^{m}(\Omega))$ where, if $\gi<\si-1$, we require
\begin{align}
\ds\frac{N(p-2)+p\si}{r}+\frac{N\si}{m}\le N(p-1+\gi)+p\si,
\label{F1}\tag{$\t{F}_{m,r}$}
\end{align}
 while if $\gi\ge\si-1$, we assume
\begin{align}
r=m=1. \label{F3}\tag{$\t{F}_1$}
\end{align}
Let us observe that we formally recover the assumption \eqref{F3} by letting  $\gamma_\infty \nearrow \sigma-1$ in \eqref{F1}.

\medskip

\begin{remark}\label{soglie}
Here we briefly comment the intervals \eqref{Q1} and \eqref{Q2}; firstly let us note that the thresholds appearing in these intervals are widely explained in Figures \ref{fig:1}, \ref{fig:2}, \ref{fig:3} and \ref{fig:4} of Section \ref{appendix} below. Here we just recall that the superlinear threshold is given by
	\begin{equation*} 
	\max\set{
		\frac{p(1+\ti)}{2},\frac{N(p-1+\ti)+p(1+\ti)}{N+2}
	}=
	\begin{cases}
	\ds
	\frac{p(1+\ti)}{2}&\t{if } p<2,\\
	\ds
	\frac{N(p-1+\ti)+p(1+\ti)}{N+2}&\t{if } p\ge 2.
	\end{cases}
	\end{equation*}
As already observed, requiring $\si>1$ in \eqref{ID1}, takes to 
\[
q>p-\frac{N(1-\ti)}{N+1},
\]
which is smaller than $\frac{p(1+\ti)}{2}$ if $\frac{2N}{N+\si}<p\le \frac{2N}{N+1}$ (see Figures \ref{fig:3}, \ref{fig:4}). \\
Moreover, we highlight that asking for $\si\ge2$ when \eqref{Q1} is in force, implies that	
\begin{align*}
&p-\frac{N(1-\ti)}{N+2}\le q<p\q \t{when}\q p>\frac{2N}{N+2},\\
&\frac{p(1+\ti)}{2}<q<p\q\t{when}\q \frac{2N}{N+\si}<p\le \frac{2N}{N+2},
\end{align*}
and here we expect finite energy solutions, i.e. solutions belonging to $L^p(0,T; W^{1,p}_0(\Omega))$.\\
We also point out that, when $\si\ge2$, then
$$
0\le\ti\le \frac{N+2}{N}\bra{q-\pare{p-\frac{N}{N+2}}},
$$
while if $1<\sigma <2$ we are asking 
$$
\frac{N+2}{N}\bra{q-\pare{p-\frac{N}{N+2}}}<\ti< \frac{N+1}{N}\bra{q-\pare{p-\frac{N}{N+1}}}.
$$
Note that, if $\ti=0$, then we recover the same $q$-thresholds which appear in the nonsingular case $g\equiv h \equiv 1$ (see \cite[Section $2$]{M}).\\
It is important to underline that, from the previous calculations, we have
\[
\ti<1,
\]
and that also implies 
\[
\si -1>\ti
\]
when the value of $\si$ is the one in \eqref{ID1}. We cannot reason in a similar way for the range \eqref{Q2}, because we will assume initial data in $L^\si(\Omega)$ for all $1<\si<2$. Without loss of generality, we set   $\si-1>\ti$ when \eqref{Q2} is in force.\\

When the value $\si$ in \eqref{ID1} is  greater than two, our results are sharp, as proved in Section \ref{sec:non}. However, we point out that even the case with $1<\si<2$ in \eqref{ID1} will be dealt with natural assumptions on the data which are widely discussed in Section \ref{appendix}.
\end{remark}

In this section we treat the case of a possibly mild singularity, which means the function $H(t,x,s,\xi)$ is admitted to blow up in such a way  the solutions always attain the trace at the boundary of the parabolic cylinder in the classical sense of Sobolev. Mathematically speaking, we are requiring 
\[
0<\tz\le 1\q\t{and}\q 0<\gz\le 1,
\]
in  \eqref{g1} and \eqref{h1}.\\

We now provide the notion of solution for this case.

\begin{defin}\label{defrin1}
	We say that a  function $u\in L^1(Q_T)$ such that $a(t,x,u,\N u) \in L^1(0,T;L^1_{\rm loc}(\Omega))$ is a distributional solution of \eqref{pb} if  
	\begin{subequations}
		\makeatletter
		\def\@currentlabel{DS}
		\makeatother
		\label{DS}
		\renewcommand{\theequation}{DS.\arabic{equation}}
	\begin{equation}\label{sr0}
\Tk{u}\in L^p(0,T; W^{1,p}_0(\Omega))\q\forall k>0,
\end{equation}
		\begin{equation}
	H(t,x,u,\nabla u)\in L^1(0,T;L^1_{\rm loc}(\Omega)),\label{sr1} 
		\end{equation}
		\begin{equation}\label{sr2}
		\begin{array}{c}
		\ds
		-\int_{\Omega} u_0\vp (0) - \iint_{Q_T}u\vp_t + \iint_{Q_T} a(t,x,u,\N u)\cdot \N \vp   
		=\iint_{Q_T}H(t,x,u,\N u)\vp  
		\end{array}
		\end{equation}
	\end{subequations}
	for every $\vp\in C_c^\infty([0,T)\times \Omega)$.
\end{defin}

We state the existence result for this section.

\begin{theorem}\label{teo} 
Assume that $a(t,x,s,\xi)$ satisfies \eqref{A1}, \eqref{A2}, \eqref{A3},  and that $H(t,x,s,\xi)$ satisfies  \eqref{H}, \eqref{h1}, \eqref{h2}, \eqref{g1}, \eqref{g2} where $\theta_0 \le 1$, $\gamma_0\le 1$ and with $q$ as in \eqref{Q}.  In particular $u_0 \in L^\sigma(\Omega)$ and $f\in L^r(0,T;L^m(\Omega))$ are nonnegative and such that:
\begin{enumerate}[i)]
\item if \eqref{Q1} holds, we assume \eqref{ID1}, \eqref{F1} when  $\gi<\si-1$, and \eqref{F3} if $\gi\ge \si-1$;
\item if \eqref{Q2} holds for any $1<\si<2$, we assume \eqref{F1} when  $\gi<\si-1$, and \eqref{F3} if $\gi \ge \si-1$.
\end{enumerate}
Then there exists at least a solution $u \in L^\infty(0,T;L^\sigma(\Omega))$ to \eqref{pb} in the sense of  Definition \ref{defrin1}.\\
In particular:
\begin{enumerate}[--]
\item if the value $\si$ in \eqref{ID1} satisfies $\si\ge2$ in case i), then $u,\,u^\beta\in L^p(0,T; W^{1,p}_0(\Omega))$;
\item if the value $\si$ in \eqref{ID1} satisfies $1<\si<2$ in case i) or case ii) holds, then $u(1+u)^{\beta-1}\in L^p(0,T; W^{1,p}_0(\Omega))$. Furthermore $|\nabla u|^{p-1} \in L^\frac{b}{p-1}(Q_T)$ while 
$u\in L^b(0,T; W^{1,b}_0(\Omega))$ if $p>1+\frac{N(2-\si)}{N+\si}$.
\end{enumerate} 
In every case
\begin{equation*}
\b=\frac{\sigma-2+p}{p}, \ \ \ b=p-\frac{N(2-\si)}{N+\si}. 
\end{equation*}
\end{theorem}

Here we resume the setting we work in and the type of results we obtain. We recall that $0\le \ti<1$. 
\begin{center}
\begin{table}[H]
\setlength{\tabcolsep}{8pt}
\renewcommand{\arraystretch}{2.4} 
\begin{tabu}{ c | c | c |c  }  
  $q$ & Assumptions on $u_0$ & Assumptions on $f$ & $\gi$
\\
\hline
\multirow{2}{*}{\eqref{Q1}}
&
\multirow{2}{*}{$u_0\in L^\si(\Omega)$ with $\si$ as in \eqref{ID1}}
&$f\in L^r(0,T;L^m(\Omega))$ with \eqref{F1} 
& $\gi<\si-1 $
\\\cline{3-4}
&
&
$f\in L^1(Q_T)$ (see \eqref{F3})
& $\gi\ge\si-1 $
\\\cline{1-4}
\multirow{2}{*}{\eqref{Q2}}
&
\multirow{2}{*}{$u_0\in L^\si(\Omega)$ for any $1<\si<2$}
&$f\in L^r(0,T;L^m(\Omega))$ with \eqref{F1} 
& $\gi<\si-1 $
\\
\cline{3-4}
& &
$f\in L^1(Q_T)$  (see \eqref{F3})
& $\gi\ge\si-1 $
\end{tabu}
\vspace*{3mm}
\caption{Global solutions when $0\le\tz,\,\gz\le 1$} 
\end{table}
\end{center}

\section{Proof of the main result}
\label{secmain}

We consider an approximation scheme of the type
\begin{equation}\label{eqn}\tag{$P_n$}
\begin{cases}
\displaystyle (\un)_t- \t{div }a(t,x,\un,\N \un) = H_n(t,x,\un, \N \un)\q  &\t{in}\ Q_T,\\
 \un=0  \q &\t{on}\ (0,T)\times \partial\Omega,\\
 \un(0,x)=u_{n,0}(x) \q & \t{in}\  \Omega,
\end{cases}
\end{equation} 
where 
$
H_n(t,x,s,\xi)=T_n(H(t,x,s,\xi))
$,
 $u_{n,0}(x)=T_n(u_0(x))$ and hence $u_{n,0}\to u_0$ strongly in $L^{\si}(\Omega)$.
By \cite{L}, it follows the existence of  a nonnegative  solution $\un$ of \eqref{eqn}  as below  
\begin{equation*}
\un\in L^p(0,T; W^{1,p}_0(\Omega))\cap L^{\infty}(Q_T)
,\,\, (\un)_t \in L^{p'}(0,T;\mathrm{W}^{-1,p'}(\Omega)),
\end{equation*}
and
\begin{equation*}
\begin{split}
\int_0^T\langle(\un)_t,\vp\rangle
+\iint_{Q_T}a(t,x,\un,\N \un)\cdot\nabla \varphi     =\iint_{Q_T}  H_n(t,x,\un,\N \un)\varphi   ,
\end{split}	
\end{equation*}
for every $\varphi\in L^p(0,T; W^{1,p}_0(\Omega))$. \\
In the sequel we will widely use the following change of variable 
$$\wn=e^{-\eta t}\un,\q \eta>0 \t{ to be fixed},$$
which makes a new zero  order term appear and this allows us to deal with forcing terms without smallness size assumptions. Then $\wn$ satisfies 
\begin{equation}\label{changeq}
(\wn)_t + \eta \wn- \t{div}(\tilde a(t,x,\wn,\N \wn)) = \widetilde{H}_n(t,x,\wn,\N \wn) ,\qq\forall\,\,n\in \mathbb{N},
\end{equation}
where the conditions \eqref{A}  still hold for 
\[
\tilde a(t,x,\wn,\N \wn)=e^{-\eta t}a(t,x,e^{\eta t}\wn, e^{\eta t}\N \wn)
\]
with different constants depending on $T$, $\eta$, and \eqref{H} is satisfied with  
\[
\widetilde{H}_n(t,x,\wn,\N\wn)=e^{-\eta t}H_n(t,x,e^{\eta t}\wn,e^{\eta t}\N\wn).
\]
We observe that $\wn$ satisfies the same boundary and initial conditions as $\un$. For sake of clarity, we omit the $\tilde{\cdot}$ in the following computations.

\medskip

We divide the proof of the a priori estimates according to the value of $\sigma$.

\subsection{Finite energy solutions}
We start analyzing the superlinear range for $q$ given by \eqref{Q1}. More particularly we deal with its subrange given by requiring $\sigma\ge 2$, namely 
\begin{align*}
&p-\frac{N(1-\ti)}{N+2}\le q<p\q \t{when}\q p>\frac{2N}{N+2},\\
&\frac{p(1+\ti)}{2}<q<p\q\t{when}\q \frac{2N}{N+\si}<p\le \frac{2N}{N+2}.
\end{align*}
We recall that, in this case,  we look for finite energy solutions. Once again and for the sake of clarity, we recall both Remark \ref{soglie} and Figures \ref{fig:1}, \ref{fig:2}, \ref{fig:3} and \ref{fig:4} in Section \ref{appendix} for a complete account on the appearing thresholds.

\begin{lemma}\label{sapp}
	Assume that $a$ satisfies \eqref{A1}, \eqref{A2}, and that $H$ satisfies  \eqref{H}, \eqref{g2}, \eqref{h2} where $q$ is in \eqref{Q1}. In particular $0\le u_0 \in L^\sigma(\Omega)$ with $\sigma\ge 2$ as in \eqref{ID1}, and $0\le f\in L^r(0,T;L^m(\Omega))$ where the couple $(r,m)$ satisfies \eqref{F1} when  $\gi<\si-1$ and \eqref{F3} if $\gi\ge \si-1$. Let  $\un$ be a solution to \eqref{eqn}. Then there exists  $\ok\in\NN$ such that we have 
		\begin{equation*} 
		\int_{\Omega} \un(t)^{\sigma} +\iint_{Q_t}|\nabla ( G_k(\un)^{\b})|^p    \le M\q \forall k\ge\ok, 
		\end{equation*}
	where
	\begin{equation*} 
	\b=\frac{\sigma-2+p}{p},
	\end{equation*}
	and  $M$ is a positive constant which does not depend on $n$.
\end{lemma}

\begin{proof}	
We multiply the equation \eqref{changeq} by the function
	\[
	G_k(\wn)^{\sigma-1}\q\t{with}\q k\ge\max\set{1,\oom},
	\] 
	and we integrate over $Q_t$ with $0<t\le T$.  Recalling \eqref{h2} and \eqref{g2} one gets 
	\begin{equation*} 
	\begin{aligned}
	&\frac{1}{\sigma}\int_{\Omega} G_k(\wn(t))^{\sigma} + \alpha \frac{\sigma-1}{\b^p}\iint_{Q_t}|\nabla ( G_k(\wn)^{\b})|^p   +\eta k\iint_{Q_t}G_k(\wn)^{\sigma-1}  
	\\ 
	&\le \frac{\bc e^{\eta(q-1)T}}{\b^q} \iint_{Q_t} |\nabla (G_k(\wn)^{\b})|^qG_k(\wn)^{\sigma-1-q(\b-1)-\theta_{\infty}}     
	\\
	& + \bc\iint_{\{f\le k^\gi\}} \frac{fG_k(\wn)^{\sigma-1}}{\wn^\gi}   +\bc\iint_{\{f>k^\gi\}} \frac{f G_k(\wn)^{\sigma-1}}{\wn^\gi}
	+\frac{1}{\si}\int_{\Omega}G_k(u_0)^{\sigma}. 
	\end{aligned}
	\end{equation*}
	Requiring $\eta>\bc/s_2$, one gets
	\[
	\bc\iint_{\{f\le k^\gi\}} \frac{fG_k(\wn)^{\sigma-1}}{\wn^\gi} \le\bc \iint_{\{f\le k^\gi\}} G_k(\wn)^{\sigma-1}\le \eta k\iint_{Q_t}G_k(\wn)^{\sigma-1}  ,
	\]
which implies
	$$
	\begin{aligned}
	&\frac{1}{\sigma}\int_{\Omega} G_k(\wn(t))^{\sigma} + \alpha \frac{\sigma-1}{\b^p}\iint_{Q_t}|\nabla ( G_k(\wn)^{\b})|^p   
	\\
	&\le \frac{ \bc e^{\eta(q-1)T}}{\b^q} \iint_{Q_t} |\nabla (G_k(\wn)^{\b})|^q G_k(\wn)^{\sigma-1-q(\b-1)-\theta_{\infty}}  
	+\bc\iint_{Q_t} f\chi_{\set{f>k^\gi}} G_k(\wn)^{\sigma-1}\wn^{-\gi}    
+\frac{1}{\si}\int_{\Omega}G_k(u_0)^{\sigma} \\
&=A+B+\frac{1}{\si}\int_{\Omega}G_k(u_0)^{\sigma}.
	\end{aligned}
	$$

	We first focus on the $A$ term. We apply H\"older's inequality twice  with exponents $\pare{\frac{p}{q},\frac{p}{p-q}}$ and $\pare{\frac{p^*}{p},\frac{N}{p}}$ so that we get 
	\begin{align*}
	A
	&\le \frac{ \bc e^{\eta(q-1)T}}{\b^q}\int_{0}^t\pare{\intO|\nabla ( G_k(\wn)^{\b})|^p}^{\frac{q}{p}}  \left( \intO G_k(\wn)^{p\b+p\frac{q-p+1-\ti}{p-q}} \right)^{\frac{p-q}{p}} \\
	&\le c_0\|G_k(\wn)\|_{L^{\infty}(0,T;L^\sigma(\Omega))}^{q-p+1-\theta_{\infty}}\iint_{Q_t}|\nabla ( G_k(\wn)^{\b})|^p    ,
	\end{align*}
	thanks also to Sobolev's inequality.\\

	As far as the $B$ term is concerned we first observe that, if $\gi\ge \sigma -1$, then 
	$$
	\begin{aligned}
	B \le \frac{\bc}{k^{\gi-\si+1}}\|f\chi_{\set{f>k^\gi}}\|_{L^1(Q_T)}\le \bc \|f\chi_{\set{f>k^\gi}}\|_{L^1(Q_T)}.
	\end{aligned}
	$$
	Hence,  
	\begin{equation}\label{stima1}
	\begin{aligned}
	&	
	\frac{1}{\sigma}\int_{\Omega} G_k(\wn(t))^{\sigma} + \alpha \frac{\sigma-1}{\b^p}\iint_{Q_t}|\nabla ( G_k(\wn)^{\b})|^p   
	\\  
	&\le c_0
	\|G_k(\wn)\|_{L^{\infty}(0,t;L^{\si}(\Omega))}^{q-p+1-\theta_{\infty}}
	\iint_{Q_t}|\nabla (G_k(\wn)^{\b})|^p 
 + \bc\|f\chi_{\set{f>k^\gi}}\|_{L^1(Q_T)} +\frac{1}{\sigma}\int_{\Omega} G_k(u_0)^{\sigma} .
	\end{aligned}
	\end{equation} 
	Otherwise, if $\gi<\sigma-1$ and by H\"older's inequalities with $(m,m'),\,(r,r')$, it yields 
	\begin{align*}
	B &\le \bc\|f\chi_{\set{f>k^\gi}}\|_{L^r(0,T;L^m(\Omega))}\|G_k(\wn)\|_{L^{r'(\si-1-\gi)}(0,T;L^{m'(\si-1-\gi)}(\Omega))}^{\sigma-1-\gi}\\
	&=\bc\|f\chi_{\set{f>k^\gi}}\|_{L^r(0,T;L^m(\Omega))}\|G_k(\wn)^{\b}\|_{L^{r'\frac{\si-1-\gi}{\b}}(0,T;L^{m'\frac{\si-1-\gi}{\b}}(\Omega))}^{\frac{\sigma-1-\gi}{\b}}.
	\end{align*}
	We go further invoking the Gagliardo-Nirenberg regularity result (see Theorem \ref{teoGN}), which states that, if  
	\[
	G_k(\wn)^{\b}\in L^{\infty}(0,T;L^{\frac{\si}{\b}}(\Omega))\cap L^p(0,T;W^{1,p}_0(\Omega)),
	\] 
	then $G_k(\wn)^{\b}\in L^{a}(0,T;L^{b}(\Omega))$ where $(b,a)$ satisfies the relation \eqref{rel}, i.e. 
	\[
	\frac{N\si}{\b b}+\frac{N(p\b -\si)+p\si}{\b a}=N.
	\]
	Moreover, it holds that
	\[
 \|G_k(\wn)^{\b}\|_{L^a(0,t;L^{b}(\Omega))}^a \le c(N,p)\|G_k(\wn)^{\b}\|_{L^{\infty}(0,t;L^{\frac{\si}{\b}}(\Omega))}^{a-p}	\iint_{Q_t}|\nabla (G_k(\wn)^{\b})|^p .
	\]
	We thus impose 
	\begin{equation}\label{ab}
	b\ge m'\frac{\si-1-\gi}{\b}\q\t{and}\q a\ge r'\frac{\si-1-\gi}{\b},
	\end{equation}
	and continue estimating $B$ as
	\begin{align}
	B&\le\, \bc\|f\chi_{\set{f>k^\gi}}\|_{L^r(0,T;L^m(\Omega))}
	\|G_k(\wn)^{\b}\|_{L^{a}(0,t;L^{b}(\Omega))}^{\frac{\si-1-\gi}{\b}}\nonumber
	\\
	&\le c\|f\chi_{\set{f>k^\gi}}\|_{L^r(0,T;L^m(\Omega))}\left(
	\|G_k(\wn)\|_{L^{\infty}(0,t;L^{\si}(\Omega))}^{\b(a-p)}	\iint_{Q_t}|\nabla (G_k(\wn)^{\b})|^p 
	\right)^{\frac{\si-1-\gi}{a\b}}\nonumber
	\\
	&\le c_1\|G_k(\wn)\|_{L^{\infty}(0,t;L^{\si}(\Omega))}^{\b(a-p)}	\iint_{Q_t}|\nabla (G_k(\wn)^{\b})|^p  +c_2\|f\chi_{\set{f>k^\gi}}\|_{L^r(0,T;L^m(\Omega))}^{\frac{a\b}{a\b-(\si-1-\gi)}},\nonumber
	\end{align}
	where the last passage is due to  Young's inequality with indices $\pare{\frac{a\b}{\si-1-\gi},\frac{a\b}{a\b-(\si-1-\gi)}}$.
	\\
	We finally get
	\begin{equation}\label{stima2}
	\begin{aligned}
	&	
	\frac{1}{\sigma}\int_{\Omega} G_k(\wn(t))^{\sigma} + \alpha \frac{\sigma-1}{\b^p}\iint_{Q_t}|\nabla ( G_k(\wn)^{\b})|^p   
	\\  
	&\le c_3\left[
	\|G_k(\wn)\|_{L^{\infty}(0,t;L^{\si}(\Omega))}^{q-p+1-\theta_{\infty}}+
	\|G_k(\wn)\|_{L^{\infty}(0,t;L^{\si}(\Omega))}^{\b(a-p)}
	\right] \iint_{Q_t}|\nabla (G_k(\wn)^{\b})|^p 
	\\ 
	&+c_2\|f\chi_{\set{f>k^\gi}}\|_{L^r(0,T;L^m(\Omega))}^{\frac{a\b}{a\b-(\si-1-\gi)}}
	+\frac{1}{\sigma}\int_{\Omega} G_k(u_0)^{\sigma} .&
	\end{aligned}
	\end{equation}

	We fix a small positive $\delta_0$  such that
	\begin{equation}\label{deo}
\begin{cases}
	\ds  c_0\pare{\si\delta_0}^{\frac{q-p+1-\ti}{\si}}= \al\frac{\si-1}{2\b^p} &\t{if }\gi\ge\si-1\q(\t{see }\eqref{stima1}),\\
	\ds
	2c_3\max\left\{
	\pare{\si\delta_0}^{q-p+1-\ti},
	\pare{\si\delta_0}^{ \b(a-p)}
	\right\}^\frac{1}{\si}= \al\frac{\si-1}{2\b^p} &\t{if }\gi<\si-1\q(\t{see }\eqref{stima2}).
\end{cases}
	\end{equation}
	Now, let $\ok$ be large enough (eventually larger than $\max\set{1,s_2}$) such that
	\begin{equation}\label{fg}
	\begin{cases}
	\ds
	\bc\norm{f\chi_{\set{f>k^\gi}}}_{L^1(Q_T)}+\frac{1}{\si}\|G_k(u_0)\|_{L^{\sigma}(\Omega)}^{\sigma}<\de_0&\t{if }\gi\ge\si-1\q(\t{see }\eqref{stima1}),\\
	\ds  c_2\norm{f\chi_{\set{f>k^\gi}}}_{L^r(0,T;L^m(\Omega))}^{\frac{a\b}{a\b-(\si-1-\gi)}}+\frac{1}{\si}\|G_k(u_0)\|_{L^{\sigma}(\Omega)}^{\sigma}
	<\de_0&\t{if }\gi<\si-1\q(\t{see }\eqref{stima2}),
\end{cases}
	\end{equation}
	for all $k\ge \ok$.\\
	Moreover let
	\begin{equation*}\label{tstar}
	T^*=\sup \set{\tau\in [0,T]: \,\,  \frac{1}{\si}\|G_k(\wn(s))\|_{L^{\sigma}(\Omega)}^{\sigma}\le\delta_0\q\forall s\le \tau } \q\forall k\ge\ok.
	\end{equation*}
	Note that $T^*>0$ since $\{\wn\}_n\subseteq C([0,T];L^{\sigma}(\Omega))$.\\
	Thus, taking into account \eqref{deo}, if $t\le T^*$, we have by \eqref{stima1} that 
	\begin{align*}
	\frac{1}{\si}\int_{\Omega} G_k(\wn(t))^{\sigma} + \alpha \frac{\sigma-1}{2\b^p }\iint_{Q_t}|\nabla ( G_k(\wn)^{\b})|^p     
	\le \bc\norm{f\chi_{\set{f>k^\gi}}}_{L^1(Q_T)}
	+\frac{1}{\si}\int_{\Omega} G_k(u_0)^{\sigma}, 
	\end{align*}
	and by \eqref{stima2} that
	\begin{align*}
	\frac{1}{\si}\int_{\Omega} G_k(\wn(t))^{\sigma} + \alpha \frac{\sigma-1}{2\b^p }\iint_{Q_t}|\nabla ( G_k(\wn)^{\b})|^p  
	\le c_2\norm{f\chi_{\set{f>k^\gi}}}_{L^r(0,T;L^m(\Omega))}^{\frac{a\b}{a\b-(\si-1-\gi)}}
	+\frac{1}{\si}\int_{\Omega} G_k(u_0)^{\sigma}. 
	\end{align*}
	Now we want to extend the previous inequalities to the whole interval $[0,T]$. Let us suppose that $T^*<T$ and set $t=T^*$. Then, by the definition of $T^*$ and the condition  \eqref{fg}, we have
	\[
	\delta_0<\de_0,
	\] 
	which is in contradiction with the definition of $T^*$. Hence, $T^*=T$, and the previous inequalities hold on the whole interval $[0,T]$.\\
	We summarize the previous result in the following inequality 
	\begin{equation}\label{m11}
	\sup_{t\in(0,T)}\int_{\Omega} G_k(\wn(t))^{\sigma} +\iint_{Q_t}|\nabla ( G_k(\wn)^{\b})|^p   \le M \q\forall k\ge\ok,
	\end{equation}
	where $M$ does not depend on $n$. \\
	\noindent The inequality in \eqref{m11} and the decomposition $\wn=\Gk{\wn}+\Tk{\wn}$ allow us to deduce that
	\[
	\sup_{t\in(0,T)}\int_{\Omega} \wn(t)^{\sigma} \le M +\ok^{\sigma}|\Omega|,
	\]
	and thus, for $T<\infty$,  $\wn$ is uniformly bounded  in $L^\infty(0,T;L^{\sigma}(\Omega))$ with respect to $n$. The desired estimate follows by setting $\wn=e^{-\eta t}\un$. This concludes the proof.
\end{proof}

\medskip

\begin{corollary}\label{todo}
	Under the assumptions of Lemma \ref{sapp},  there exists $\hat{k}\in\NN$ such that we have
	\begin{equation*} 
	\iint_{Q_T}|\nabla \Gk{\un}|^p   \le M\q\forall k\ge\hat{k},
	\end{equation*}
 where $M$ is a positive constant which does not depend on $n$.
\end{corollary}

\begin{proof}
	We take $\Gk{\un}$,  $k\ge s_2$,  as  test function in \eqref{eqn}, providing the following estimate 
	$$
	\begin{aligned}
	\frac{1}{2}\int_{\Omega} \Gk{\un(T)}^{2} +  \alpha \iint_{Q_T}|\nabla \Gk{\un}|^p   
	&\le   \bc\iint_{Q_T} \frac{|\nabla \Gk{\un}|^q}{\un^\ti} \Gk{\un}  
	+\bc\iint_{Q_T} \frac{f}{\un^\gi} \Gk{\un}     + \frac{1}{2}\int_{\Omega} \Gk{\un(0)}^{2}\\
&=A+B+\frac{1}{2}\int_{\Omega} \Gk{\un(0)}^{2}.
	\end{aligned}
	$$
	Since $0\le\ti<1$,  H\"older's inequality implies that   
	\begin{align*}
	A\le  \bc\pare{\iint_{Q_T}|\nabla \Gk{\un}|^p }^{\frac{q}{p}}\pare{
		\iint_{Q_T} \Gk{\un}^{\frac{p}{p-q}(1-\ti)}
	}^{\frac{p-q}{p}} .
	\end{align*}
	Since by definition of $\si$
	\[
	\frac{p}{p-q}(1-\ti)=p\frac{N+\si}{N},
	\]
	which is the Gagliardo-Nirenberg interpolation exponent (see Theorem \ref{teoGN}) for
	\[
	L^\infty(0,T;L^\si(\Omega))\cap L^p(0,T;W^{1,p}_0(\Omega)),
	\]
we have
\[
\iint_{Q_T} \Gk{\un}^{\frac{p}{p-q}(1-\ti)}=\iint_{Q_T} \Gk{\un}^{p\frac{N+\si}{N}}\le c\norm{\Gk{\un}}_{L^\infty(0,T;L^\si(\Omega))}^\frac{p\si}{N}\iint_{Q_T}|\N \Gk{\un}|^p.
\]
Hence, we can estimate the $A$ term as 
	\[
	A\le c \|G_k(\un)\|_{L^{\infty}(0,T;L^{\si}(\Omega))}^{q-p+1-\ti}\iint_{{Q_T}}|\nabla  G_k(\un)|^p.
	\]
	As far as $B$ is concerned, we just say that
	\[
	B\le c\norm{f}_{L^1(Q_T)}
	\]
	when $\gi\ge1$.\\
	Now let $0\le\gi<1$. We apply the H\"older inequality obtaining
	\[
	B\le \|f\|_{L^r(0,T;L^m(\Omega))}
	\|G_k(\un)\|_{L^{r'(1-\gi)}(0,T;L^{m'(1-\gi))}(\Omega))}^{1-\gi}.
	\]
	We require
	\[
	a \ge r'(1-\gi)\q\t{and}\q b\ge m'(1-\gi),
	\]
	for $a,\,b$ such that
	\[
	\frac{N\si}{ b}+\frac{N(p -\si)+p\si}{ a}=N,
	\]
	in order to apply Theorem \ref{teoGN} in
	\[
	L^\infty(0,T;L^\si(\Omega))\cap L^p(0,T;W^{1,p}_0(\Omega)),
	\] 
and obtain
	\begin{align*}
	B &\le c\|f\|_{L^r(0,T;L^m(\Omega))}\left(
	\|G_k(\un)\|_{L^{\infty}(0,T;L^{\si}(\Omega))}^{a-p}\iint_{{Q_T}}|\nabla  G_k(\wn)|^p
	\right)^{\frac{1-\gi}{a}}\nonumber
	\\
	&\le c_1\|G_k(\un)\|_{L^{\infty}(0,T;L^{\si}(\Omega))}^{a-p}\iint_{{Q_T}}|\nabla  G_k(\wn)|^p +c_2\|f\|_{L^r(0,T;L^m(\Omega))}^{\frac{a}{a-(1-\gi)}},
	\end{align*}
	where the last passage is due to Young's inequality with indices $\pare{\frac{a}{1-\gi},\frac{a}{a-(1-\gi)}}$.
	
	Now, for $k$ large enough, $\|G_k(\un)\|_{L^{\infty}(0,T;L^{\si}(\Omega))}$ is as small as one needs, thanks to Lemma \ref{sapp}. This concludes the proof.
\end{proof}

Now we need some estimates near the origin, so we work on the truncations of $\un$.

\begin{lemma}\label{tk}
	Under the assumptions of Lemma \ref{sapp}, let us also assume that \eqref{g1} and \eqref{h1} hold with $\tz\le 1$ and $\gz\le 1$. Then 
	\begin{equation}\label{disTk}
	\iint_{Q_T}|\nabla T_{k}(\un)|^p      \le M\qq\forall k>0,
	\end{equation} 
	where $M$ is a positive constant which does not depend on $n$.  
\end{lemma}

\begin{proof}
	
	We first prove that  $|\N T_{\ok}(\un)|$ is uniformly bounded in $L^p(Q_T)$ with respect to $n$ for $k\ge \hat{k}$ (where $\hat{k}>s_2$ has been defined in Corollary \ref{todo}). Then, we obtain the result by monotonicity. To this aim, we choose $T_{k}( \un)$ as test function  in \eqref{eqn}, getting
	\begin{align*}
	&\int_{\Omega} \Theta_{k}( \un(T)) +
	\al \iint_{Q_T}|\nabla  T_{k}( \un)|^p   \\
	&\le    \max_{[0,k]}s\,g(s)\iint_{\{ \un\le k\}} |\nabla  \un|^q 
	+\max_{[0,k]}s\,h(s)\iint_{{\{ \un\le k\}}} f + \bc \iint_{\{ \un> k\}} |\nabla  \un|^q T_{k}( \un) \un^{-\ti } \\
	&+ \sup_{[k,+\infty)}h(s)\iint_{{\{ \un>{k}\}}} f T_{k}( \un)    + \int_{\Omega} \Theta_{k}(u_0) \\
&=A+B+ C+D+\int_{\Omega} \Theta_{k}(u_0) ,
	\end{align*}
	where $\Theta_{k}(s)=\int_0^{s}T_{k}(v)\,dv$. We point out that \eqref{g1}, \eqref{h1} and $\tz,\,\gz\le 1$ ensure that $\max_{[0,k]}s\,g(s),\,\max_{[0,k]}s\,h(s)$  are finite.\\
	We  estimate the $A$ term as follows:
	\begin{align*}
	A&\le 
	\max_{[0,k]}s\,g(s) \iint_{\{ \un\le k\}} |\N T_{k}( \un)|^q   \le \frac{ \al}{2}\iint_{Q_T}|\nabla T_{k}( \un)|^p+c(k)
	\end{align*}
	We also have  
	\begin{align*}
	B+D\le   
	c(k)\|f\|_{L^1(Q_T)}.
	\end{align*}
	Now, we estimate $C$ as 
	\begin{align*}
	C&\le \bc k^{1-\ti}\iint_{\{ \un> k\}} |\nabla  \un|^q, 
	\end{align*}
	which is bounded by Corollary \ref{todo}. 
	Thus,  the estimate \eqref{disTk} follows  by monotonicity for any $k>0$.

\end{proof}

\begin{remark}\label{stimaenergiafinita}
 We underline that Lemma \ref{sapp} gives us the uniform boundedness of $\un$ in $L^{p\b\frac{N+\frac{\si}{\b}}{N}}(Q_T)$ with respect to $n$ (recall that $\b=\frac{\si+p-2}{p}$). 
Indeed, by Lemma \ref{sapp}, $G_k(\un)^{\beta}$  is bounded with respect to $n\in\NN$ in
\[
L^\infty(0,T;L^\frac{\si}{\b}(\Omega))\cap L^p(0,T;W^{1,p}_0(\Omega))\q\forall k\ge \ok.
\]
Then,  by Theorem \ref{teoGN}, the claim follows. Note also that $p\b\frac{N+\frac{\si}{\b}}{N}>p$. 
Furthermore, Lemma \ref{todo} and Lemma \ref{tk} imply that $|\nabla \un|$ is actually bounded in $L^p(Q_T)$ with respect to $n$.
\end{remark}

We continue with  the following regularity lemma:

\begin{lemma}\label{pot}
		Under the assumptions of Lemma \ref{tk}, it holds that
		\begin{equation*} 
		\int_{\Omega} \un(t)^{\sigma} +\iint_{Q_t}|\nabla \un^{\b}|^p    \le M,
		\end{equation*}
	where $M$ is a positive constant which does not depend on $n$.
\end{lemma}
\begin{proof}
	It follows from Lemma \ref{sapp} that it is sufficient to show that $|\N (T_{\ok}(\un)^\beta)|$ is bounded in $L^p(Q_T)$ with respect to $n$ in order to conclude the proof. 
	Hence we test the equation in \eqref{eqn} with $T_{\ok}(\un)^{\sigma-1}$, so that we have 
		\begin{align*}
	\ds
	&\alpha\frac{ \sigma-1}{\beta^p}\iint_{Q_T}|\nabla ( T_{\ok}(\un)^{\beta})|^p\\
&	\le \sup_{[0,+\infty)} \pare{g(s)T_{\ok}{(s)}^{\si-1}}
	\iint_{Q_T}|\N\un|^q
	+\sup_{[0,+\infty)} \pare{h(s)T_{\ok}({s})^{\si-1}}\iint_{Q_T}f +\int_{\Omega}\Theta_{\ok}(u_0),
	\end{align*}
where $\Theta_{k}(s)=\int_0^{s}T_{\ok}(v)^{\si-1}\,dv$.\\
The right-hand is bounded by a constant, depending on $\ok$, thanks to Corollary \ref{todo} and Lemma \ref{tk}.
\end{proof}

\subsection{Infinite energy solutions}

This section deals with the case in which the initial datum $u_0$ does not belong to $L^2(\Omega)$ and solutions with infinite energy are expected to exist. As for $q$, this means that we are in \eqref{Q1} range for small values of $q$ or in the \eqref{Q2} one, i.e.  
\begin{align*}
&\max\set{\frac{p(1+\ti)}{2},\frac{N(p-1+\ti)+p(1+\ti)}{N+2}}<q<p-\frac{N(1-\ti)}{N+2}\q\t{for}\q p>\frac{2N}{N+1},\\
&\frac{p(1+\ti)}{2}<q\le p-\frac{N(1-\ti)}{N+2}\q\t{for}\q  \frac{2N}{N+2}<p\le \frac{2N}{N+1}.
\end{align*}
We refer to Figures \ref{fig:1}, \ref{fig:2}, and \ref{fig:3} for further comments on the above ranges.

\medskip

We start with the following lemma:

\begin{lemma}\label{sappinfinite}
	
Assume that $a$ satisfies \eqref{A1}, \eqref{A2}, and that $H$ satisfies \eqref{H} with \eqref{g2}, \eqref{h2} where $q$ is as in \eqref{Q1} or \eqref{Q2}. In particular $u_0 \in L^\sigma(\Omega)$ and $f\in L^r(0,T;L^m(\Omega))$ are nonnegative and such that:
\begin{enumerate}[i)]
	\item if \eqref{Q1} holds and the value $\si$ is as in \eqref{ID1} and it verifies $1<\si<2$, we assume \eqref{F1} when  $\gi<\si-1$, and \eqref{F3} if $\gi\ge \si-1$;
	\item if \eqref{Q2} holds, we let $1<\sigma<2$ and we assume \eqref{F1} when  $\gi<\si-1$, and \eqref{F3} if $\gi \ge \si-1$.
\end{enumerate}
Let $\un$ be a solution to \eqref{eqn}. Then, there exists  $\ok\in\NN$ such that  
\begin{equation}\label{disGk2infinite}
\int_{\Omega} \un(t)^{\sigma} +\iint_{Q_t}|\nabla \Gk{\un}|^p(1+\Gk{\un})^{\si-2}   \le M\q\forall k\ge\ok,
\end{equation}
where $M$ is a positive constant which does not depend on $n$. 	
\end{lemma}

\begin{proof}\textit{Case i)}\\
\noindent
We multiply  \eqref{changeq} by 
\[
\vp_{\eps}(\Gk{\wn})=\pare{\eps+\Gk{\wn}}^{\si-1}-\eps^{\si-1}\q\t{with}\q 0<\eps<1,\,k>\oom.
\]
Then, integrating over $Q_t$ for $0< t\le T$ and thanks to the assumptions \eqref{A1}, \eqref{H}, \eqref{g2}, \eqref{h2}, we have: 
\begin{equation}\label{main}
\begin{aligned}
&\int_{\Omega} \Theta_\varepsilon(G_k(\wn(t))) + \al(\si-1)\iint_{Q_t}|\N \Phi_\varepsilon(G_k(\wn))|^p 
\\
\ds
&\le \bc  \iint_{Q_t} \frac{|\nabla G_k(\wn)|^q}{\wn^\ti}\vp_{\eps}(\Gk{\wn})+
\bc\iint_{\{ f>k^{\gi}\}} \frac{f}{\wn^{\gi}} \vp_{\eps}(\Gk{\wn})  +   \frac{2^{\si-1}\eps^{\si}|\Omega|}{\si}+\frac{2^{\si-1}}{\si}\int_{\Omega} G_k(u_0)^\sigma\\
&=A+B+\frac{2^{\si-1}\eps^{\si}|\Omega|}{\si}+\frac{2^{\si-1}}{\si}\int_{\Omega} G_k(u_0)^\sigma,
\end{aligned} 
\end{equation}
where we have set
\[
\Theta_\vare(s)=\int_0^s\vp_{\eps}(z)\,dz\q\t{and}\q \Phi_\varepsilon(s)=\frac{1}{\b}\pare{\pare{\eps+s}^\b-\eps^\b}.
\]
Let us also observe that the integral involving the forcing term where $\set{f\le k^\gi}$ has been treated as in Lemma \ref{sapp}. We also used that 
$$\int_{\Omega} \Theta_\vare(G_k(u_0))
\le \frac{2^{\si-1}\eps^{\si}|\Omega|}{\si}+\frac{2^{\si-1}}{\si}\int_{\Omega} G_k(u_0)^\sigma.$$
\\
 Now, let us observe that the definition of $\vp_\varepsilon $ allows us to estimate  $A$  as follows:
$$
\begin{aligned}
A&\le\bc\b{^q}\iint_{Q_t}\frac{|\N \Phi_\varepsilon(G_k(\wn))|^q}{\wn^\ti}(\eps+\Gk{\wn})^{-\frac{q}{p}(\si-2)}\vp_{\eps}(\Gk{\wn})
\\
&\le \bc\b{^q}\iint_{Q_t} |\N \Phi_\varepsilon(G_k(\wn))|^q(\eps+\Gk{\wn})^{q-p+1-\ti +\beta(p-q)} ,
\end{aligned}
$$
where we used that, for $k$ large, one has $w_n\ge \eps + \Gk{\wn}$.
\\ 
Now, by  H\"older's inequality with indices $\left(\frac{p}{q},\frac{p^*}{p-q},\frac{N}{p-q}\right)$, Sobolev's embedding and the definition of $\si$, we obtain that
\begin{align*}
A&\le c\int_0^t\|\eps+\Gk{\wn}\|_{L^{\sigma}(\Omega)}^{q-p+1-\ti}\ \|\nabla \Phi_\varepsilon (G_k(\wn))\|_{L^p(\Omega)}^p\\
&\le c\|\eps+\Gk{\wn}\|_{L^\infty(0,t;L^{\sigma}(\Omega))}^{q-p+1-\ti}\iint_{{Q_t}}|\nabla \Phi_\varepsilon (G_k(\wn))|^p\\
&\le c_1\|\Gk{\wn}\|_{L^\infty(0,t;L^{\sigma}(\Omega))}^{q-p+1-\ti}\iint_{{Q_t}}|\nabla \Phi_\varepsilon (G_k(\wn))|^p+ c\eps^{q-p+1-\ti}\iint_{{Q_t}}|\nabla \Phi_\varepsilon (G_k(\wn))|^p.
\end{align*}

As far as the $B$ term is concerned, we consider again both the cases in which either $\gi\ge \si-1$ or $\gi<\si-1$. \\
If   $\gi\ge \si-1$, we can estimate $B$ as 
\begin{align*}
B
&\le \bc\iint_{{\set{\wn>k}}} f\chi_{\set{f>k^\gi}}\wn^{-\gi}\pare{\eps+\Gk{\wn}}^{\si-1} \\
&\le \bc\iint_{{\set{\wn>k}}} f\chi_{\set{f>k^\gi}}\wn^{\si-1-\gi}\\
&\le \frac{\bc}{k^{\gi-\si+1}}\iint_{{Q_t}} f\chi_{\set{f>k^\gi}},
\end{align*}
and then 
\[
B\le c_2\norm{f\chi_{\{f>k^\gi\}}}_{L^1(Q_T)}.
\]
Now, let $\gi<\si-1$ and observe that 
\begin{equation*} 
(k+x)^{-\gi}\vp_{\eps}(x)\le c\,\Phi_\eps(x)^{\frac{\sigma-1-\gi}{\b}}
\end{equation*}
for some $c>0$ independent from $\eps$ and for  $k\ge \eps^{-\frac{(\gi(2-\si)-p(\si-1))_+}{\gi(\si+p-2)}}$. Then, this estimate and two applications of H\"older's inequality with $(m,m')$ and $(r,r')$ imply that we can deal with $B$ as follows:
\begin{align*}
B &\le c\iint_{\{ f>k^\gi\}}f\, \Phi_\eps(\Gk{\wn})^{\frac{\sigma-1-\gi}{\b}} \\
&\le c\norm{f\chi_{\set{f>k^\gi}}}_{L^r(0,T;L^m(\Omega))} \|\Phi_\varepsilon(G_k(\wn))\|_{L^{r'\frac{\sigma-1-\gi}{\b}}(0,t;L^{m'\frac{\sigma-1-\gi}{\b}}(\Omega))}^{\frac{\sigma-1-\gi}{\b}}.
\end{align*}
We apply the inequality \eqref{disGN} in Theorem \ref{teoGN} to the function $\Phi_\vare(\Gk{\wn})$ in the space
\[
L^\infty(0,T;L^\frac{\si}{\b}(\Omega))\cap L^p(0,T;W^{1,p}_0(\Omega)),
\]
so that, recalling \eqref{F1} too, 
we get
\begin{align*}
\ds
B&\le c\|  f\chi_{\{ f>k^\gi\}}\|_{L^r(0,T;L^m(\Omega))} \|\Phi_\varepsilon(G_k(\wn))\|_{L^{a}(0,t;L^{b}(\Omega))}^{\frac{\sigma-1-\gi}{\b}}\\
&\le c\|\eps+G_k(\wn)\|_{L^{\infty}(0,t;L^{\si}(\Omega))}^{\b(a-p)}\iint_{{Q_t}}|\nabla \Phi_\varepsilon (G_k(\wn))|^p
+ c_3\norm{f\chi_{\set{f>k^\gi}}}_{L^r(0,T;L^m(\Omega))}^{\frac{a\b}{a\b-(\si-1-\gi)}}\\
&
\le c_4\|G_k(\wn)\|_{L^{\infty}(0,t;L^{\si}(\Omega))}^{\b(a-p)}\iint_{{Q_t}}|\nabla \Phi_\varepsilon (G_k(\wn))|^p
+
c\eps^{\b(a-p)}\iint_{{Q_t}}|\nabla \Phi_\varepsilon (G_k(\wn))|^p 
+ c_3\norm{f\chi_{\set{f>k^\gi}}}_{L^r(0,T;L^m(\Omega))}^{\frac{a\b}{a\b-(\si-1-\gi)}},
\end{align*}
where $(a,b)$ satisfies \eqref{ab}, i.e.
\[
	b\ge m'\frac{\si-1-\gi}{\b}\q\t{and}\q a\ge r'\frac{\si-1-\gi}{\b}.
\]
Analogously to the proof of Lemma \ref{sapp}, we set $\de_0$ and $\eps$ such that
\begin{equation*}\label{grad}
c\eps^{q-p+1-\ti} + c_1(\si\de_0)^{\frac{p-q}{N}}+c\eps^{\b(a-p)} +c_4 (\si\de_0)^{\frac{\b(a-p)}{\si}}\le \frac{\al(\si-1)}{2},
\end{equation*}
and we also set $\ok$ large enough such that, for all $k\ge\ok$, we have 
\begin{equation}\label{cond}
	\begin{cases}
	\ds
\frac{2^{\si-1}}{\si} 
\norm{G_k(u_0)}_{L^\si(\Omega)}^\sigma+c_2\norm{f\chi_{\set{f>k^\gi}}}_{L^1(Q_T)} <\frac{\de_0}{4}&\t{if }\gi\ge\si-1 ,\\[3mm]
	\ds
\frac{2^{\si-1}}{\si} \norm{G_k(u_0)}_{L^\si(\Omega)}^\sigma+c_3\norm{f\chi_{\set{f>k^\gi}}}_{L^r(0,T;L^m(\Omega))}^{\frac{a\b}{a\b-(\si-1-\gi)}}
<\frac{\de_0}{4}&\t{if }\gi<\si-1,
\end{cases}
	\end{equation}
We also define 
	\begin{equation*}\label{tstar2}
	T^*=\sup \set{\tau\in [0,T]: \,\,  \frac{1}{\si}\|G_k(\wn(s))\|_{L^{\sigma}(\Omega)}^{\sigma}\le\delta_0\q\forall s\le \tau } \q\forall k\ge\ok.
	\end{equation*}
Now let us observe that
\begin{equation}\label{new1}
\begin{aligned}
\Theta_\varepsilon(G_k(\wn(t)))= 
\frac{(\eps+ G_k(w_n))^\sigma}{\sigma} - \frac{\eps^\sigma}{\sigma} -\eps^{\sigma-1}G_k(w_n) 
&\ge 
\frac{G_k(w_n)^\sigma}{\sigma} - \frac{\eps^\sigma}{\sigma} -\eps^{\sigma-1}G_k(w_n)
\\
&\ge 
\frac{G_k(w_n)^\sigma}{2\sigma}-\frac{\eps^\si}{\si}\pare{1+(\si-1)2^\frac{1}{\si-1}},
\end{aligned}
\end{equation}
and that
\begin{align}\label{new2}
|\nabla \Phi_\varepsilon(G_k(\wn))|^p&\ge  |\nabla \Phi_1(G_k(\wn))|^p= |\N G_k(\wn)|^p(1+G_k(\wn))^{\sigma-2}.
\end{align}
We now use the estimates on the terms A and B, as well as \eqref{new1} and \eqref{new2}, in \eqref{main}. We thus obtain that
\begin{align*}
	\frac{1}{2\si}\int_{\Omega} G_k(\wn(t))^{\sigma} + \frac{\al(\si-1)}{2}\iint_{Q_t}|\N \Phi_\varepsilon(G_k(\wn))|^p \nonumber
&\le \frac{2^{\si-1}}{\si}\norm{G_k(u_0)}_{L^\si(\Omega)}^\sigma + c_3\norm{f\chi_{\set{f>k^\gi}}}_{L^r(0,T;L^m(\Omega))}^{\frac{a\b}{a\b-(\si-1-\gi)}} \\
&+
\frac{\eps^\si|\Omega|}{\si}\pare{2^{\si-1}+1+(\si-1)2^\frac{1}{\si-1}}
\end{align*}
 if $\gi<\si-1$, and 
\begin{align*}
	\frac{1}{2\si}\int_{\Omega} G_k(\wn(t))^{\sigma} + \frac{\al(\si-1)}{2}\iint_{Q_t}|\N \Phi_\varepsilon(G_k(\wn))|^p \nonumber
&\le \frac{2^{\si-1}}{\si}\norm{G_k(u_0)}_{L^\si(\Omega)}^\sigma + c_2\norm{f\chi_{\set{f>k^\gi}}}_{L^1(Q_T)}\\
&+\frac{\eps^\si|\Omega|}{\si}\pare{2^{\si-1}+1+(\si-1)2^\frac{1}{\si-1}}
\end{align*}
if  $\gi\ge\si-1$.\\
Now if we suppose that $T^*<T$ and we evaluate the above inequalities for $t=T^*$, then we get in both cases
\[
\frac{\de_0}{2}\le \frac{\de_0}{4}+\frac{\eps^\si|\Omega|}{\si}\pare{2^{\si-1}+1+(\si-1)2^\frac{1}{\si-1}},
\]
which gives a contradiction if one considers $\eps$ suitably small. Note that taking $\eps$ smaller makes $k$ grow, so \eqref{cond} still holds. \\
Hence,  
\begin{align*}
&\sup_{t\in (0,T)}\int_{\Omega} G_k(\wn(t))^\sigma + \iint_{Q_T} |\N G_k(\wn)|^p(1+G_k(\wn))^{\sigma-2} \le M,
\end{align*}
where $M$ is positive and independent of $n$. 
We recover \eqref{disGk2infinite} recalling the relation between $\wn$ and $\un$, and this concludes the proof of Case i).
\\

\noindent
\textit{Case ii)}\\
\noindent
This case can be dealt with exactly as before, except for the A term in \eqref{main}. Thus, we come back to A and estimate again with  H\"older's inequality with indices $\left(\frac{p}{q},\frac{p^*}{p-q},\frac{N}{p-q}\right)$ and  Sobolev's embedding, getting
\begin{align*}
A&\le c\int_0^t\norm{\pare{\eps+\Gk{\wn}}^{\frac{N(q-p+1-\ti)}{p-q}}}_{L^{1}(\Omega)}^{\frac{p-q}{N}}\ \|\nabla \Phi_\varepsilon (G_k(\wn))\|_{L^p(\Omega)}^p.
\end{align*}
We now take advantage of the fact that $\si>N(q-p+1-\ti)/(p-q)$, and apply again the H\"older inequality with $\pare{\frac{\si(p-q)}{N(q-p+1-\ti)},\frac{\si(p-q)}{\si(p-q)-N(q-p+1-\ti)}}$, so that
\begin{align*}
A&\le  c|\Omega|^{\frac{\si(p-q)-N(q-p+1-\ti)}{N\si}}\|\eps+\Gk{\wn}\|_{L^\infty(0,t;L^{\sigma}(\Omega))}^{q-p+1-\ti}\iint_{{Q_t}}|\nabla \Phi_\varepsilon (G_k(\wn))|^p.
\end{align*}
The inequality above allows us to reason as in the case i), of course, with slight changes.
\end{proof}

\medskip

\begin{remark}\label{unp1}
Note that, also in this case, $\un^{p\b\frac{N+\frac{\si}{\b}}{N}}$ (recall $\b=\frac{\si+p-2}{p}$) is uniformly bounded with respect to $n$ in $L^1(Q_T)$ and that  $p\b\frac{N+\frac{\si}{\b}}{N}>p-1$.
Indeed,
$$\ds
|\nabla ((1+ G_k(\un))^{\b-1} G_k(\un))|^p \le c|\N G_k(\un)|^p(1+G_k(\un))^{\sigma-2},
$$	
which implies, thanks to Lemma \ref{sappinfinite}, that 
$(1+G_k(\un))^{\beta-1}G_k(\un)$ is bounded with respect to $n\in\NN$ in
\[
L^\infty(0,T;L^\frac{\si}{\b}(\Omega))\cap L^p(0,T;W^{1,p}_0(\Omega))\q\forall k\ge \ok.
\]
Then one concludes by Theorem \ref{teoGN}.
\end{remark}
 
For this section, we will refer to $\ok$ as the one defined in Lemma  \ref{sappinfinite}. Without loss of generality, we assume that $\ok>1$. 

\medskip

\begin{corollary}\label{todoinfinite}
Under the assumptions of Lemma \ref{sappinfinite},  we have
\begin{equation*} 
\iint_{Q_T}|\nabla \Gk{\un}|^b   \le M\q\forall k\ge \ok,
\end{equation*}
where
$$\ds b=p-\frac{N(2-\si)}{N+\si},$$ 
and $M$ is a positive constant which does not depend on $n$.
\end{corollary}
\begin{proof}
Let us observe that 
\begin{align}\label{stimatodo}
\iint_{Q_T}|\N \Gk{\un}|^b  &\le \left( \iint_{Q_T}|\N \Gk{\un}|^p(1+\Gk{\un})^{\si-2}   \right)^{\frac{b}{p}} \left(\iint_{Q_T} (1+\Gk{\un})^{\frac{b(2-\si)}{p-b}}  \right)^{\frac{p-b}{p}}.
\end{align}
It follows from Lemma \ref{sappinfinite} that the first integral is uniformly bounded with respect to $n\in\NN$ for all $k\ge \ok$. Moreover, by Remark \ref{unp1}, $\un^{p\b\frac{N+\frac{\si}{\b}}{N}}$ is uniformly bounded with respect to $n$ in $L^1(Q_T)$.
\\
Hence, the right-hand of \eqref{stimatodo} is bounded by a constant which does not depend on $n$ if 
 $$\frac{b(2-\si)}{p-b}=p\b\frac{N+\frac{\si}{\b}}{N},$$
 which gives  
\begin{equation*} 
 b=p-\frac{N(2-\si)}{N+\si}.
\end{equation*}
\end{proof}

 \medskip

\begin{remark}[On the  parameter $b$] \label{b}
Note that, in every case, the  value of $b$ in Corollary \ref{todoinfinite} satisfies $b>p-1$.
\end{remark}

\medskip

\begin{corollary}\label{L1infinito}
Under the assumptions of Lemma \ref{sappinfinite}, we have 
\begin{equation*}
|||\N \un|^q \un^{-\ti}\chi_{\set{\un>k}}||_{L^{\rho}(Q_T)}\le M\q\forall k>\ok,
\end{equation*}
for some $1<\rho<\frac{p}{q}$ where the constant $M$ does not depend on $n$.
\end{corollary}
\begin{proof}
We take advantage of \eqref{disGk2infinite} and we estimate as
\begin{equation}
\begin{aligned}\label{stimalinf}
\iint_{\set{\un>k}}\frac{|\N \un|^{q\rho}}{\un^{\ti\rho}}
&\le \iint_{Q_T}|\N \Gk{\un}|^{q\rho}\frac{(1+\Gk{\un})^{\frac{q\rho(\si-2)}{p}}}{(1+\Gk{\un})^{\frac{q\rho(\si-2)}{p}+\ti\rho}}\\
&\le \pare{
\iint_{Q_T}\frac{|\N \Gk{\un}|^{p}}{(1+\Gk{\un})^{2-\si}}
}^{\frac{q\rho}{p}}
\pare{\iint_{Q_T}(1+\Gk{\un})^{\frac{p}{p-q\rho}\pare{\frac{q\rho(2-\si)}{p}-\ti\rho}}}^\frac{p-q\rho}{p}.
\end{aligned}
\end{equation}
Once again, from Remark \ref{unp1}, the first integral in \eqref{stimalinf}  is uniformly bounded with respect to $n\in\NN$  for all $k\ge \ok$. The same remark implies that $\un^{p\b\frac{N+\frac{\si}{\b}}{N}}$ is uniformly bounded with respect to $n$ in $L^1(Q_T)$. Hence, the right-hand of \eqref{stimalinf} is bounded with respect to $n$ if one requires
\[
\frac{p}{p-q\rho}\pare{\frac{q\rho(2-\si)}{p}-\ti\rho}=p\b\frac{N+\frac{\si}{\b}}{N},
\]
namely
\[
\rho=\frac{N(\si+p-2)+p\si}{N(q-\ti)+q\si}.
\]
Note that under our assumptions $\rho>1$.
\end{proof}

\medskip

\begin{lemma}\label{tk2}

Under the assumptions of Lemma \ref{sappinfinite}, let us also assume that \eqref{g1} and \eqref{h1} hold with $\tz\le 1$ and $\gz\le 1$. Then
\begin{equation*} 
\iint_{Q_T}|\nabla T_{k}(\un)|^p      \le M\qq\forall k>0,
\end{equation*} 
where $M$ is a positive constant which does not depend on $n$.  
\end{lemma}
\begin{proof}
The proof takes advantage of Corollary \ref{L1infinito} and it is the very close to the one of Lemma \ref{tk}.
\end{proof}

\medskip

Now we are in position to show the following global bound:

\begin{corollary}\label{todo2}
Under the assumptions of Lemma \ref{tk2}, let $\un$ be a solution to \eqref{eqn}. Then
\begin{equation*} 
\iint_{Q_T}|\nabla ((1+ \un)^{\b-1} \un)|^p \le M,
\end{equation*}
where $M$ is a positive constant which does not depend on $n$.
\end{corollary}
\begin{proof}
We estimate
\begin{align*}
\ds
\iint_{Q_T}|\nabla ((1+ \un)^{\b-1} \un)|^p &\le c\iint_{Q_T}|\N \un|^p(1+\un)^{\sigma-2} \\
&
\le c\iint_{Q_T\cap \{\un>\ok\}}|\N G_{\ok}(\un)|^p(1+G_{\ok}(\un))^{\si-2} +c\iint_{Q_T\cap\{\un\le \ok\}}|\N T_{\ok}(\un)|^p ,
\end{align*}
and the right-hand side is bounded with respect to $n$ thanks to Lemmas \ref{sappinfinite} and \ref{tk2}. The proof is concluded.
\end{proof}

\subsection{Proof of Theorem \ref{teo} }

\begin{lemma}\label{propstimaL1}
Under the assumptions of Lemma \ref{tk} or Lemma \ref{tk2}, let us also assume that \eqref{A3} holds. Then
	\begin{equation}\label{stimaL1}
	\iint_{Q_T} H_n(t,x,\un, \N \un)\varphi \le M,
	\end{equation}
	for any nonnegative $\vp\in C^1_c([0,T)\times \Omega)$ and where the constant $M$  does not depend on $n$.\\ Moreover, there exists a locally integrable weakly differentiable function $u$ such that 
\[
\un\to u\q\t{strongly in}\q L^\mu(Q_T)\q\t{for any}\q\mu<\si,
\]
and 
\[
\N\un\to\N u\q\t{a.e. in}\q Q_T,
\]
 as $n\to \infty$. 
\end{lemma}
\begin{proof}
\noindent
\textit{Uniform estimate of $H_n$.}\\
\noindent
By Corollary \ref{todo}, Lemma \ref{tk} for the finite energy case, and
Corollary \ref{todoinfinite}, Lemma \ref{tk2} for the infinite energy one, $u_n$ is bounded in $L^\infty(0,T;L^\sigma(\Omega))$ and $\un^{p-1},\,|\N \un|^{p-1}$ are bounded in $L^{\frac{b}{p-1}}(Q_T)$  with respect to $n$ (see Remarks \ref{stimaenergiafinita}, \ref{unp1}, and \ref{b}). Hence, we take a nonnegative $\varphi\in C^1_c([0,T)\times \Omega)$ as a test function in \eqref{eqn} and, thanks to \eqref{A2}, we deduce that
	\begin{align*}\label{stimal1}
	\iint_{Q_T} H_n(t,x,\un, \N \un)\varphi &\le - \int_{Q_T} \un\varphi_t + c \int_{Q_T} (|\nabla \un|^{p-1}+\un^{p-1}+|\ell|)|\nabla \varphi|\le M,
	\end{align*}
	for some $M$ which does not depend on $n$.\\

\noindent
\textit{Strong convergence of $\un\to u$ in $L^\mu(Q_T)$, $\mu<\si$.}\\
\noindent
 The almost everywhere convergence of $u_n$ towards $u$ in $Q_T$ follows by reasoning as in the proof of \cite[Theorem 2.1]{Ptrunc}. 
The strong convergence of $\un\to u$ in $L^\mu(Q_T)$, $\mu<\si$, simply follows from the uniform boundedness of $\un$ in $L^\infty(0,T;L^\si(\Omega))$.\\

\noindent
\textit{Almost everywhere convergence of $\N\un\to \N u$.}\\
\noindent
We reason as in Proposition $3.14$ of \cite{dpp}; hence we just sketch the main differences. Firstly, for $k>0$, we denote by $T_k(u)_\nu$ the unique nonnegative solution to 
\begin{equation*}
\begin{cases}
\displaystyle \frac{1}{\nu}(T_k(u)_\nu)_t + T_k(u)_\nu = T_k(u) ,\\
T_k(u)_\nu(0,x)=T_k(u_{\nu,0}(x)),
\end{cases}
\end{equation*}
where $u_{\nu,0}(x) = T_{\nu}(u_0(x))$. The strategy is to take 
$$\psi_{n,h,\nu} = T_{2k}(\un-T_h(\un) + T_k(\un) - (T_k(u)_\nu)\varphi,$$
for $h>2k$ and $0\le \varphi \in C^1_c(\Omega)$, as a test function in \eqref{eqn}, in order to show that
$$\lim_{n\to \infty}\iint_{{Q_T}} \left(a(t,x, T_k(\un),\nabla T_k(\un)) - a(t,x, T_k(\un),\nabla T_k(u))\right) \cdot \nabla (T_k(\un)- T_k(u)) = 0,$$ 
and then to apply \cite[Lemma $5$]{bmp88} in order to deduce the almost everywhere convergence of the gradients in $Q_T$.
\\
There are a few differences with respect to \cite{dpp}. Firstly, they deal with solutions which satisfy the following property:
\begin{equation}\label{condrin}
	\lim_{h\to\infty} \limsup_{n\to\infty} \iint_{\{h<\un<h+k\}} |\nabla \un|^p = 0.
\end{equation}
Moreover here the operator in divergence form also depends on $\un$ but this is not really an issue and the proof works also in this case. Finally, we have the right-hand which is only bounded in $L^1(0,T; L^1_{\rm loc}(\Omega))$ and we need to show that 
\begin{equation}\label{conv1}
\lim_{h\to\infty}	\lim_{\nu \to\infty}\lim_{n\to\infty} \iint_{Q_T} H_n(t,x,\un,\nabla \un) \psi_{n,h,\nu} = 0.  
\end{equation}
 Hence, in order to repeat the proof contained in \cite{dpp}, we just need to show that \eqref{condrin} and \eqref{conv1} hold. 
 
 \medskip
 
 We start with \eqref{condrin}. We consider the following function 
 \begin{equation*}\label{funren}
 \begin{split}\!
 \displaystyle
 S(s)=
 \begin{cases}
 0 \ \ &\t{if}\q s\le h, \\
 \displaystyle s-h \ \ &\t{if}\q h <s< h+k, \\
 k \ \ &\t{if}\q s\ge h+k,
 \end{cases}
 \end{split}
 \end{equation*}
 and we test \eqref{eqn} by $S(\un)$ with $h>s_2$, yielding to
 \begin{align*}
 \alpha\iint_{\{h<\un<h+k\}} |\nabla \un|^p &\le k\iint_{{\{ \un>h\}}} \left(\frac{|\nabla \un|^q}{\un^\ti} + \frac{f}{\un^\gi}  + u_0\right)\le  k\iint_{{\{ \un>h\}}} \frac{|\nabla \un|^q}{\un^\ti} +\frac{k}{h^\gi}\iint_{{\{ \un>h\}}}  f+k\iint_{{\{ \un>h\}}}  u_0.
 \end{align*}
Corollary \ref{L1infinito} implies that the integral on the right-hand involving the gradient term is equi-integrable in $n$. The same holds for the last integral, since $u_0\in L^\si(\Omega)$.
Then, with  $f\in L^1(Q_T)$, one  can take first $n\to \infty$ and then $h\to \infty$ deducing \eqref{condrin}.

\medskip

In order to show \eqref{conv1} we split as follows ($0<\de<s_2$ small enough):
 \begin{equation}\label{conv0}
 \iint_{Q_T} H_n(t,x,\un,\nabla \un) \psi_{n,h,\nu} = \iint_{\{\un\le \de\}} H_n(t,x,\un,\nabla \un) \psi_{n,h,\nu} +\iint_{\{\un>\de\}} H_n(t,x,\un,\nabla \un) \psi_{n,h,\nu}.
 \end{equation}
 We observe that
 \begin{equation*}\label{conv2}
 \iint_{\{\un\le \de\}} H_n(t,x,\un,\nabla \un) \psi_{n,h,\nu} \le 2\iint_{\{\un\le \de\}} H_n(t,x,\un,\nabla \un)\un\varphi \le c\de,
 \end{equation*}
thanks to \eqref{stimaL1} for some $c$ independent on $n,\nu, h$.
 \\
As far as the integral over the set $\set{\un>\de}$ is concerned, we take advantage of \eqref{H}, to say that
\begin{align*}
 \iint_{\{\un> \de\}} H_n(t,x,\un,\nabla \un) \psi_{n,h,\nu}\le \iint_{\{\un> \de\}} g(\un)|\N \un|^q \psi_{n,h,\nu}
+ \sup_{s\in[\de,+\infty)}h(s) \iint_{\{\un> \de\}} f \psi_{n,h,\nu}.
\end{align*}
It follows from Corollary \ref{L1infinito} and from the fact that $T_k(\un)$ is bounded in $L^p(0,T; W^{1,p}_0(\Omega))$ with respect to $n$ that 
\begin{equation}\label{conv3}
\iint_{\{\un> \de\}} g(\un)|\N \un|^q \psi_{n,h,\nu} \le c \left(\iint_{Q_T} T_{2k}^{\rho'}(\un-T_h(\un) +T_k(\un) - T_k(u)_\nu)\right)^{\frac{1}{\rho'}},
\end{equation}
where $c$ does not depend on $n,\nu,h$. Moreover, the right-hand of the previous simply tends to zero as $n,\nu,h$ tend to infinity. 
\\Also the second term on the right-hand of \eqref{conv3} converges to zero, since $f$ belongs (at least) to $L^1(Q_T)$.  
All these computations imply that taking the limit first as $n,\nu,h\to\infty$ and then $\de\to 0$ in \eqref{conv0}, one gets \eqref{conv1}. 
\end{proof}

\medskip

\begin{proof}[Proof of Theorem \ref{teo}]
Let $\un$ be a solution to \eqref{eqn}. First, we recall that, by Lemma \ref{propstimaL1}, $\un$ converges almost everywhere to some function $u$ as $n\to \infty$. 
By Lemmas \ref{tk} and \ref{tk2} one gains that $T_k(u)\in L^p(0,T; W^{1,p}_0(\Omega))$ for any $k>0$ which is condition \eqref{sr0}. \\
In order to recover \eqref{sr2}, we want to pass to the limit every term in the approximating formulation
\begin{equation}\label{solnn}
\begin{split}
\intO u_{n,0}\vp(0)-\int_{Q_T}\un\vp_t
+\iint_{Q_T}a(t,x,\un,\N \un)\cdot\nabla \varphi     =\iint_{Q_T}  H_n(t,x,\un,\N \un)\varphi   ,
\end{split}	
\end{equation}
where  $\vp\in C^1_c([0,T)\times \Omega)$. 

\medskip

 One can simply pass to the limit the first and second term in \eqref{solnn} recalling the definition of $u_{n,0}$ and the strong convergence of $\un$ in $L^\mu(Q_T)$, $\mu<\si$, due to Lemma \ref{propstimaL1}. 

\medskip

As far as the principal part is concerned, it follows from \eqref{A2}, by Corollary \ref{todo}, Lemma \ref{tk}, Remark \ref{stimaenergiafinita} for the finite energy case, and Corollary \ref{todoinfinite}, Lemma \ref{tk2}, Remark \ref{unp1} for the infinite energy case, that $a(t,x,\un,\N \un)$ is strongly compact in $L^1(Q_T)$. Moreover, by the almost everywhere  convergence of the gradients in Lemma \ref{propstimaL1},  $a(t,x,\un,\N \un)$ converges to $a(t,x,u,\N u)$ as $n\to\infty$.

\medskip

It remains to pass to the limit the possibly singular term $H_n(t,x,\un,\N\un)$. If both $g(0) <\infty$ and $h(0)<\infty$, then $H(t,x,0,\xi)$ is finite and then the proof is done with a simple application of Vitali's Theorem, since we have Remark \ref{stimaenergiafinita} (when $\si\ge2$) and Corollary \ref{L1infinito}, Lemma \ref{tk2} (when $\si<2$).
\\ Hence, we suppose that $H(t,x,s,\xi)$  is actually blowing up as $s\to0$. We take a nonnegative $\varphi\in C^1_c([0,T)\times\Omega)$ and  $\delta >0$ such that $\delta\not\in \{\eta: |u=\eta|>0\}$, which is possible since it is a countable set. We have  
	\begin{equation}\label{rhs}
	\int_{Q_T} H_n(t,x,\un, \nabla \un)\varphi = \int_{Q_T\cap \{\un\le \delta\}}H_n(t,x,\un, \nabla \un)\varphi + \int_{Q_T\cap \{\un> \delta\}}H_n(t,x,\un, \nabla \un)\varphi.
	\end{equation}
	We want to pass first $n\to \infty$ and then $\delta\to 0$ in the previous equality. In the second term on the right-hand of \eqref{rhs}, we can simply pass to the limit by Lebesgue's Theorem as $n\to \infty$ (we are far from zero), obtaining
	\begin{equation*} 
	\lim_{n\to \infty}\int_{Q_T\cap \{\un> \delta\}}H_n(t,x,\un, \nabla \un)\varphi= \int_{Q_T\cap \{u> \delta\}}H(t,x,u, \nabla u)\varphi.
	\end{equation*}
	Now we observe that an application of Fatou's Lemma in \eqref{stimaL1} with respect to $n$ gives that $H(t,x,u,\nabla u)\in L^1(0,T; L^1_{\rm loc}(\Omega))$ which gives \eqref{sr1} (recall that $H$ is a Carath\'eodory function). Then, an application of Lebesgue's Theorem yields  
	\begin{equation*} 
	\lim_{\delta\to 0}\lim_{n\to \infty}\int_{Q_T\cap \{\un> \delta\}}H_n(t,x,\un, \nabla \un)\varphi= \int_{Q_T\cap \{u>0\}}H(t,x,u, \nabla u)\varphi.
	\end{equation*}	
	Let us now prove that the first term on the right-hand of \eqref{rhs} tends to zero as $n\to\infty$ and $\delta\to 0$.  
We  define the function $V_\de(s)$ as
\begin{equation*}
\begin{split}\!
\displaystyle
V_{\delta}(s)=
\begin{cases}
1 \ \ &\t{if}\q s\le \delta, \\
\displaystyle\frac{2\delta-s}{\delta} \ \ &\t{if}\q \de <s< 2\delta, \\
0 \ \ &\t{if}\q s\ge 2\delta,
\end{cases}
\end{split}
\end{equation*}
and take $V_\de(\un)\vp$ as a test function in \eqref{eqn} obtaining  
	\begin{equation*}\label{limn1}
	\begin{aligned}
	\int_{Q_T\cap \{\un\le \delta\}}H_n(t,x,\un, \nabla \un)\varphi &\le \int_0^T \langle(\un)_t,V_\delta(\un)\varphi\rangle + \int_{Q_T}a(t,x,\un,\N\un)\cdot\nabla \varphi V_{\delta}(\un) 
	\\
	& -\frac{\al}{\delta} \int_{Q_T\cap\set{\de<\un<2\de}}|\nabla \un|^{p}\varphi
	\\
	&
	\le -\int_\Omega \Phi_\delta (\un)\varphi_t + \int_{Q_T}a(t,x,\un,\N\un)\cdot\nabla \varphi V_{\delta}(\un) \\
& \le C\delta + \int_{Q_T}a(t,x,\un,\N\un)\cdot\nabla \varphi V_{\delta}(\un) ,
	\end{aligned}
	\end{equation*}
	where $\Phi_\delta (s) = \int_0^s V_\delta(t) \ dt$. 
 Let us also underline that, in the previous estimate, we employed Lemma $7.1$ of \cite{DP}. Indeed, even if $V_\delta(0)\ne 0$, as  requested in that lemma, its proof can be repeated taking into account the fact that $\vp\in C_c^1([0,T)\times\Omega)$.\\
Hence, we apply Fatou's Lemma twice, first in $n$ and then in $\de$, obtaining
	\begin{equation}\label{limn33}
	\begin{aligned}
\int_{Q_T\cap \{u=0\}}H(t,x,u, \nabla u)\varphi \le \int_{Q_T \cap \{u=0\}}a(t,x,u,\N u)\cdot\nabla \varphi   = 0.
	\end{aligned}
	\end{equation}
Let us observe that the last equality follows by the fact that $a$ is a Carath\'eodory function which, together with \eqref{A1}, implies that
\begin{equation*}
a(t,x,0,0)=0\q\t{for a.e. }(t,x)\in Q_T.
\end{equation*}
 In particular, \eqref{limn33} gives that the first term in \eqref{rhs} tends to zero as $n\to\infty$ and $\de\to0$.\\
	Hence we have shown that 
	\begin{equation*}
\lim_{n\to\infty}\eqref{rhs}= \int_{Q_T\cap\set{u>0}}H(t,x,u, \nabla u)\varphi\overset{\eqref{limn33}}{=} \int_{Q_T}H(t,x,u, \nabla u)\varphi.
	\end{equation*}	
This proves that \eqref{sr2} holds and that $u$ is a solution to \eqref{pb}. 

As far as the inequalities contained in the statement of Theorem \ref{teo} are concerned,  they are a consequence of Lemma \ref{pot} and Lemma \ref{todo2}.
\end{proof}

\section{Existence result in the strong singular case}
\label{sec:strong}

In this section we briefly treat the strong singular case, i.e. 
\[
\tz>1\q\t{and/or}\q \gz>1.
\]
We begin by adapting the notion of solution we consider in this case, namely we recover the Dirichlet condition in a weaker sense than in Definition \ref{defrin1}. 
We have the following: 
\begin{defin}\label{defrinstrong}
	A function $u\in L^1(Q_T)$ such that $a(t,x,u,\N u) \in L^1(0,T;L^1_{\rm loc}(\Omega))$  is a distributional solution of \eqref{pb} if  
		\begin{equation}
		\Tk{G_\varepsilon(u)} \in L^p(0,T; W^{1,p}_0(\Omega)) \ \ \forall k,\,\eps>0, \label{strong1} 
		\end{equation}
		\begin{equation}
		H(t,x,u,\nabla u)\in L^1(0,T; L_{\rm loc}^1(\Omega)),\label{strong2} 
		\end{equation}
		\begin{equation}\label{strong3}
		\begin{array}{c}
		\ds
		-\int_{\Omega} u_0\vp (0) - \iint_{Q_T}u\vp_t + \iint_{Q_T} a(t,x,u,\N u)\cdot \N \vp   
		=\iint_{Q_T}H(t,x,u,\N u)\vp  
		\end{array}
		\end{equation}
	for every $\vp\in C_c^\infty([0,T)\times \Omega)$.
\end{defin}
\begin{remark}
	We underline that condition \eqref{strong1} is where we recover the Dirichlet boundary condition; this is quite natural when dealing with singular problems.
 For example, in the stationary case \cite{CST} for $H(x,s)=f s^{-\gz}$, the request $G_\eps(u)\in W^{1,p}_0(\Omega)$ (for any $\eps>0$) is employed   and the authors prove that the solution, under suitable assumptions, is unique.
\end{remark}

\medskip

We now are ready to state the existence theorem, which is analogous to Theorem \ref{teo}.
 
\begin{theorem}\label{teostrong} 
Assume that $a(t,x,s,\xi)$ satisfies \eqref{A1}, \eqref{A2}, \eqref{A3},  and that $H(t,x,s,\xi)$ satisfies  \eqref{H}, \eqref{h1}, \eqref{h2}, \eqref{g1}, \eqref{g2}  with either $\theta_0 > 1$ and/or $\gamma_0> 1$  and with $q$ as in \eqref{Q}.  In particular $u_0 \in L^\sigma(\Omega)$ and $f\in L^r(0,T;L^m(\Omega))$ are nonnegative and such that:
\begin{enumerate}[i)]
\item if \eqref{Q1} holds, we assume \eqref{ID1}, \eqref{F1} when  $\gi<\si-1$, and \eqref{F3} if $\gi\ge \si-1$;
\item if \eqref{Q2} holds for any $1<\si<2$, we assume \eqref{F1} when  $\gi<\si-1$, and \eqref{F3} if $\gi \ge \si-1$.
\end{enumerate}
Then there exists at least a solution $u \in L^\infty(0,T;L^\sigma(\Omega))$ to \eqref{pb} in the sense of  Definition \ref{defrinstrong}.\\
In particular:
\begin{enumerate}[--]
\item when the value $\si$ in \eqref{ID1} satisfies $\si\ge2$ in case i), then $G_\eps(u),\,G_\eps(u)^\beta\in L^p(0,T; W^{1,p}_0(\Omega))$ for any $\eps >0$ ;
\item when the value $\si$ in \eqref{ID1} satisfies $1<\si<2$ in case i) or case ii) holds, then $G_\eps(u)(1+G_\eps(u))^{\beta-1}\in L^p(0,T; W^{1,p}_0(\Omega))$ for any $\eps>0$. Furthermore,
$|\nabla G_\eps(u)|^{p-1} \in L^{\frac{b}{p-1}}(Q_T)$ for any $\eps>0$, while  
 $G_\eps(u)\in L^b(0,T; W^{1,b}_0(\Omega))$ for any $\eps>0$   if $p>1+\frac{N(2-\si)}{N+\si}$.
\end{enumerate} 
In every case
\begin{equation*}
\b=\frac{\sigma-2+p}{p}, \ \ \ b=p-\frac{N(2-\si)}{N+\si}. 
\end{equation*}
\end{theorem}

\begin{center}
\begin{table}[H]
\setlength{\tabcolsep}{8pt}
\renewcommand{\arraystretch}{2.4} 
\begin{tabu}{ c | c | c |c  }  
  $q$ & Assumptions on $u_0$ & Assumptions on $f$ & $\gi$
\\
\hline
\multirow{2}{*}{\eqref{Q1}}
&
\multirow{2}{*}{$u_0\in L^\si(\Omega)$ with $\si$ as in \eqref{ID1}}
&$f\in L^r(0,T;L^m(\Omega))$ with \eqref{F1} 
& $\gi<\si-1 $
\\\cline{3-4}
&
&
$f\in L^1(Q_T)$ (see \eqref{F3})
& $\gi\ge\si-1 $
\\\cline{1-4}
\multirow{2}{*}{\eqref{Q2}}
&
\multirow{2}{*}{$u_0\in L^\si(\Omega)$ for any $1<\si<2$}
&$f\in L^r(0,T;L^m(\Omega))$ with \eqref{F1} 
& $\gi<\si-1 $
\\
\cline{3-4}
& &
$f\in L^1(Q_T)$ (see \eqref{F3})
& $\gi\ge\si-1 $
\end{tabu}
\vspace*{3mm}
\caption{Local solutions when $\max\set{\tz,\,\gz}> 1$} 
\end{table}
\end{center}

In order to prove the previous theorem we employ once again the scheme of approximation \eqref{eqn}. 

\begin{remark}\label{saploc}
In this section we are treating problem \eqref{pb} where $H(t,x,u,\N u)$ satisfies \eqref{H} with  \eqref{g1} and \eqref{h1} with $\tz$ and/or $\gz$ possibly greater than one. Hence, this means that   all the estimates contained in Section \ref{secmain} concerning $G_k(\un)$ continue to hold, i.e. 
\begin{enumerate}[--]
\item when $\si\ge2$, $\un$ and $\Gk{\un}^\b$ are uniformly bounded, respectively, in $L^\infty(0,T;L^\si(\Omega))$ and $L^p(0,T;W^{1,p}_0(\Omega))$ with respect to $n$ from Lemma \ref{sapp}. Moreover, Corollary \ref{todo} gives that $\Gk\un$ is uniformly bounded in $L^p(0,T;W^{1,p}_0(\Omega))$ with respect to $n$;
\item when $\si<2$,  Lemma \ref{sappinfinite} and Remark \ref{unp1} give us that  $\un$ and $\pare{1+\Gk{\un}}^{\b-1}\Gk{\un}$ are uniformly bounded, respectively, in $L^\infty(0,T;L^\si(\Omega))$ and $L^p(0,T;W^{1,p}_0(\Omega))$ with respect to $n$. Furthermore, 
Corollaries \ref{todoinfinite} and \ref{L1infinito}  continue to hold.
\end{enumerate}
Hence we just need to prove a priori estimates in order to deduce \eqref{strong1} and some local estimates on truncations of $\un$. 
\end{remark}

\begin{lemma}\label{sappstrong}
	Assume that $a$ satisfies \eqref{A1}, \eqref{A2}, and that $H$ satisfies  \eqref{H}, \eqref{g2}, \eqref{h2}   with $q$ as in \eqref{Q1}. In particular $u_0 \in L^\sigma(\Omega)$ with $\sigma\ge 2$ as in \eqref{ID1} and $f\in L^r(0,T;L^m(\Omega))$ are nonnegative and such that the couple $(r,m)$ satisfies \eqref{F1} when  $\gi<\si-1$ and \eqref{F3} if $\gi\ge \si-1$. Let $\un$ be a solution to \eqref{eqn} then 
\begin{equation*}
		\iint_{Q_T}|\nabla T_k(G_\varepsilon(\un))|^p   \le M \q \forall k,\,\eps>0, 
\end{equation*}	
where $M>0$ is a constant which does not depend on $n$.
\end{lemma}
\begin{proof}
We test \eqref{eqn} by $T_k(G_\varepsilon(\un))$, $k>\max(\hat{k}, \ok)$,  obtaining
\begin{equation*}
\begin{aligned}
\int_{\Omega} \Phi_{k,\eps}(\un(T)) + \alpha \iint_{Q_T}|\nabla T_k(G_\varepsilon(\un))|^p     
&\le k \iint_{Q_T}  g(\un) |\nabla G_\varepsilon(\un)|^q      + k \sup_{s\in (\varepsilon,+\infty)} h(s)\norm{f}_{L^1(Q_T)} 
+\int_{\Omega}\Phi_{k,\eps}(u_0), 
\end{aligned}
\end{equation*}	
where $\Phi_{k,\eps}(s)=\int_0^sT_k(G_\eps(v))\,dv$.
 Then, using also \eqref{g2}, one has
\begin{equation*}
\begin{aligned}
\alpha \iint_{Q_T}|\nabla T_k(G_\varepsilon(\un))|^p     
&\le k \sup_{s\in (\varepsilon,+\infty)} g(s) \iint_{Q_T} |\nabla T_k(G_\varepsilon(\un))|^q + \overline{c}k\iint_{Q_T} \un^{-\ti}|\nabla G_{\varepsilon+k}(\un)|^q  
\\
&+ k \sup_{s\in (\varepsilon,+\infty)} h(s)\norm{f}_{L^1(Q_T)} 
+k\norm{u_0}_{L^1(\Omega)}.
\end{aligned}
\end{equation*}	
Hence, applying  Young's inequality, we deduce
\begin{equation*} 
\begin{aligned}
\frac{\alpha}{2} \iint_{Q_T}|\nabla T_k(G_\varepsilon(\un))|^p     
&\le C + k \overline{c}\iint_{Q_T} \un^{-\ti}|\nabla G_{\varepsilon+k}(\un)|^q   
+ k \sup_{s\in (\varepsilon,+\infty)} h(s)\norm{f}_{L^1(Q_T)} 
+k\norm{u_0}_{L^1(\Omega)}. 
\end{aligned}
\end{equation*}	
Since $k>\max(\hat{k}, \ok)$, we deduce from Corollary \ref{todo} and Corollary \ref{L1infinito} that the right-hand of the previous equality is less than or equal to a constant which does not depend on $n$. The proof concludes by monotonicity. 
\end{proof}
Now we prove the local estimates on truncations of $\un$.
\begin{lemma}\label{tkstrong}
	Under the assumptions of Lemma \ref{sappstrong}, it holds that
	\begin{equation*}
	\iint_{Q_T}|\nabla T_k(\un)|^p\varphi^p   \le M\q\forall k>0, 
	\end{equation*}	
	where $0\le\varphi \in C^1_c([0,T) \times \Omega)$ and $M>0$ is a constant which does not depend on $n$.
\end{lemma}
\begin{proof}
We consider a nonnegative $\varphi \in C^1_c([0,T) \times \Omega)$ and we multiply \eqref{eqn} by $(T_k(\un)-k)\varphi^p$. Integrating in space, one gets
\begin{equation*} 
\begin{aligned}
& \langle (\un)_t, (T_k(\un)-k)\varphi^p \rangle + \alpha \int_{\Omega}|\nabla T_k(\un)|^p\varphi^p 
+ p\int_{\Omega} a(t,x,\un,\nabla \un)\cdot \nabla \varphi \varphi^{p-1}(T_k(\un)-k) \le 0.
\end{aligned}
\end{equation*}		
Then, integrating in time too and estimating the integral involving the divergence term, we have
\begin{equation}
\begin{aligned}
\alpha \iint_{Q_T}|\nabla T_k(\un)|^p\varphi^p &\le 
pk\iint_{Q_T} |a(t,x,\Tk{\un},\nabla \Tk{\un})||\nabla \varphi| \varphi^{p-1} +\int_\Omega S_k(u_0)\varphi(0,x) + 
\iint_{Q_T} S_k(\un) \varphi_t\\
&\le  cpk\varepsilon \iint_{Q_T} |\nabla T_k(\un)|^p\varphi^{p} + cpkc_\eps\iint_{Q_T} |\nabla \varphi|^p + cpk\iint_{Q_T} (T_k(\un)^{p-1}+ l) |\nabla \varphi| \varphi^{p-1} 
\\ &\q+ \int_\Omega S_k(u_0)\varphi(0)^p + 
p\iint_{Q_T} S_k(\un)\varphi^{p-1} \varphi_t,\label{local4}
\end{aligned}
\end{equation}
thanks also to  Young's inequality, and
where $S_k(s) = \int_0^s (T_k(z)-k)$. Hence taking $\varepsilon$ sufficiently small and observing that $|S_k(s)|\le Ck$, then one gets that the right-hand of \eqref{local4} is less than or equal to a constant $M$ which is independent of $n$.
\end{proof}

\medskip 

\begin{proof}[Proof of Theorem \ref{teostrong}] 
 Let $\un$ be a solution to \eqref{eqn}. 
Reasoning as in Lemma \ref{propstimaL1}, 
we deduce the strong $L^\mu$-convergence of $\un\to u$ for $\mu<\si$, and
 we get the uniform boundedness in $L^1(0,T; L^1_{\rm loc}(\Omega))$ of $H(t,x,\un,\N\un)$ which also gives \eqref{strong2}. Now, we can localize the  proof of Lemma \ref{stimaL1}, in order to deduce the almost everywhere  convergence of the gradients.\\
The weak formulation \eqref{strong3} can be obtained exactly as in the proof of Theorem \ref{teo}. Condition \eqref{strong1} follows from Lemma \ref{sappstrong}. This concludes the proof.  
\end{proof}

\section{A nonexistence result }
\label{sec:non}

The aim of this section is proving that the hypothesis assumed in the growth range case \eqref{Q1} are sharp in order to have an existence result, at least in the finite energy case.\\

For the sake of simplicity, we consider the following model  problem
\begin{equation}\label{D}
\begin{cases}
\ds u_t-\D_p u=g(u)|\N u|^q &\t{in }Q_T,\\
 u=0&\t{on }(0,T)\times\partial\Omega,\\
u(0,x)=u_0(x)& \t{in }\Omega,
\end{cases}
\end{equation}
where 
$$g(s)=
\begin{cases}
1 \ \ &\text{if }0\le s\le 1,\\
s^{-\ti} \ \ &\text{if }s>1,
\end{cases}
$$
and $p>2$.\\
The goal is to show that there exists some initial datum
\begin{equation}\label{IDNE}  
u_0\in L^{\eta}(\Omega)\q\t{with}\q 2\le\eta<\si=\frac{N(q-p+1-\ti)}{p-q},
\end{equation}
such that the problem \eqref{D} does not admit any solution $u$ such that 
\begin{equation}\label{solne1}
u\in L^\infty(0,T;L^\eta(\Omega))\cap L^p(0,T;W_0^{1,p}(\Omega)) \text{ and }u^{\frac{\eta+p-2}{p}}\in  L^p(0,T;W_0^{1,p}(\Omega)).
\end{equation}
Note that we are no longer dealing with solutions to \eqref{D} such that $u^{\frac{\sigma-2+p}{p}}\in L^p(0,T; W^{1,p}_0(\Omega))$ since $\sigma>\eta$.\\
Our choice for the initial datum is given by  
\begin{equation}\label{u0}
	u_0(x)=|x|^{-\frac{N}{\eta}+\om}\chi_{\{|x|<1\}}
\end{equation}
for $\om>0$ sufficiently small, so that $u_0$ fulfills \eqref{IDNE}. 
\\
Hence we work by contradiction assuming that there exists a solution $u$ to \eqref{D}, where $u_0$ is given by \eqref{u0}, and such that \eqref{solne1} holds.
\\We follow the idea contained in \cite[Subsection $3.2$]{BASW}: we show integral inequalities for $u$ from above and for  $U$ from below, where $U=U(t,x)$ is the solution to
\begin{equation*}\label{UU}
\begin{cases}
U_t-\D_p U=0 &\text{in}\,\,Q_T,\\
U=0 &\text{on}\,\,(0,T)\times\partial\Omega,\\
U(0,x)=u_0(x)  & \text{in}\,\, \Omega.
\end{cases}
\end{equation*}
Having $u=U=0$ on $(0,T)\times\partial\Omega$, and also $u(0,x)=U(0,x)=u_0(x)$ in $\Omega$, allows us to apply standard comparison results for $u_t-\D_p u$ (we quote, for instance, \cite{Di}) and deduce that $U\le u$.\\

We start showing a lower bound for $U$; indeed it follows from the Harnack inequality (see \cite[Chapter $VI$ Paragraph $8$]{Di}) and from \eqref{u0} that 
\begin{equation*}
c_1r^{-\frac{N}{\eta}+\om}\le \fint_{B_r}u_0(x) \le c_2\bra{
\pare{\frac{r^p}{t}}^\frac{1}{p-2}
+\pare{\frac{t}{r^p}}^\frac{N}{p}
\pare{\inf_{y\in B_r}U(t,y)}^{\frac{\lm}{p}}
},
\end{equation*}
where $t>0$, $r<1 $, $B_r=B_r(0)$, $c_2=c_2(N,p)$ and $\lm =p(N+1)-2N$. Hence this implies that
\[
r^{-\frac{N}{\eta}+\om}\le c
\bra{
\pare{\frac{r^p}{t}}^\frac{1}{p-2}
+\pare{\frac{t}{r^p}}^\frac{N}{p}
\pare{\inf_{y\in B_r}U(t,y)}^{\frac{\lm}{p}}
}.
\]
If $r$ is small enough such that
\begin{equation}\label{start}
t\gg r^{p+\frac{(N-\om \eta)(p-2)}{\eta}},
\end{equation}
we have $r^{-\frac{N}{\eta}+\om}\gg \left(\frac{r^p}{t}\right)^\frac{1}{p-2}$ and then we obtain
\[
r^{-\frac{N}{\eta}+\om}\le 
c\pare{\frac{t}{r^p}}^\frac{N}{p}
\pare{\inf_{y\in B_r}U(t,y)}^{\frac{\lm}{p}}.
\]
Hence
\begin{equation*}
\inf_{y\in B_r}U(t,y)\ge c t^{-\frac{N}{\lm}}r^{\frac{pN}{\lm}\frac{\eta-1}{\eta} +p\frac{\om}{\lm}},
\end{equation*}
which implies
\begin{equation}\label{infU}
\int_{B_r}U(t,y) \ge c t^{-\frac{N}{\lm}}r^{\frac{pN}{\lm}\frac{\eta-1}{\eta} +p\frac{\om}{\lm}+N}.
\end{equation}

We now look for a bound from above for the solution of \eqref{D}.
We assume that $\ti<\eta-1(<\si-1)$. Indeed, if  $\ti\ge\eta-1(<\si-1)$, the problem becomes sublinear in the gradient, i.e., no compatibility conditions among the data and the growth rate of the gradient are needed in order to get existence of solutions.  Then, a solution $u$ fulfills
$u^{\frac{\eta+q-1-\ti}{q}}\in L^q(0,T;W^{1,q}_0(\Omega))$, namely
\begin{equation}\label{nonex}
\iint_{Q_T} |\N u|^q u^{\eta-1-\ti} < \infty.
\end{equation}
Indeed, after a simple density argument, one can test the equation in \eqref{D} by $T_k(u)^{\eta-1}$, $k>1$, and getting rid of the nonnegative terms on the right-hand, obtaining
\begin{equation*}
\begin{aligned}
	\iint_{Q_T} |\N G_1(u)|^q T_k(u)^{\eta-1}u^{-\ti} &\le \iint_{Q_T} g(u)|\N u|^q T_k(u)^{\eta-1}\\
 &\le c\pare{\int_\Omega T_k(u(T))^\eta +  \iint_{Q_T} |\nabla T_k^\frac{\eta-2+p}{p}(u)|^p }
	\\
	&\le c\pare{\int_\Omega u(T)^\eta +  \iint_{Q_T} |\nabla u^\frac{\eta-2+p}{p}|^p}, 
\end{aligned}
\end{equation*} 
which, under the assumptions on $u$, is finite. Moreover, an application of Fatou's Lemma as $k\to\infty$ gives that 
\begin{equation*}
\begin{aligned}
\iint_{Q_T} |\N G_1(u)|^q u^{\eta-1-\ti} <\infty. 
\end{aligned}
\end{equation*} 
We also observe that 
\begin{equation*}
\begin{aligned}
\iint_{Q_T} |\N T_1(u)|^q u^{\eta-1-\ti} <\infty, 
\end{aligned}
\end{equation*} 
by the assumption on $\ti$. This shows that \eqref{nonex} holds.
Then, there exists at least a sequence $\{t_j\}_j$ satisfying $t_j\to 0$ such that 
\begin{equation*} 
\int_{\Omega}|\N u(t_j)|^qu(t_j)^{\eta-1-\ti} \le c\norm{\N \pare{u(t_j)^{\frac{\eta+q-1-\ti}{q}}}}_{L^q(\Omega)}^q\le \frac{c}{t_j}. 
\end{equation*}
Then, applying  H\"older's inequality with exponents $\left(q^*\frac{\eta+q-1-\ti}{q}, \left(q^*\frac{\eta+q-1-\ti}{q} \right)'\right)$, we have
\begin{equation}\label{u}
\begin{split}
\int_{B_r}u(t_j,y)  &\le c \|u(t_j)\|_{L^{q^*\frac{\eta+q-1-\ti}{q}}(\Omega)}
r^{N-\frac{N-q}{\eta+q-1-\ti}}\\
&\le c\|\N (u(t_j)^{\frac{\eta+q-1-\ti}{q}})\|_{L^q(\Omega)}^{\frac{q}{\eta+q-1-\ti}}\,
r^{N-\frac{N-q}{\eta+q-1-\ti}}\\
& \le ct_j^{-\frac{1}{\eta+q-1-\ti}}r^{N-\frac{N-q}{\eta+q-1-\ti}}.
\end{split}
\end{equation}
Now, since $U\le u$, we gather \eqref{infU} and \eqref{u} in order to obtain 
\begin{equation*}
t_j^{-\frac{N}{\lm}}r^{N+\frac{pN}{\lm}\frac{\eta-1}{\eta} +p\frac{\om}{\lm}}\le\int_{B_r}U(t_j,y) \le\int_{B_r}u(t_j,y) \le ct_j^{-\frac{1}{\eta+q-1-\ti}}r^{N-\frac{N-q}{\eta+q-1-\ti}}
\end{equation*}
from which
\begin{equation*} 
t_j^{\frac{1}{\eta+q-1-\ti}-\frac{N}{\lm}}r^{\frac{pN}{\lm}\frac{\eta-1}{\eta} +p\frac{\om}{\lm}+\frac{N-q}{\eta+q-1-\ti}}\le c.
\end{equation*}
We recall \eqref{start} and set $r=\eps t_j^{\frac{\eta}{p\eta+N(p-2)-\om\eta(p-2)}}$, where $0<\eps\ll 1$, obtaining
\begin{equation}\label{dis222}
t_j^{\vp}\le c(\eps)
\end{equation}
for $\vp=\vp(\eta)$ defined by
\begin{align*}
\vp&=-\frac{N}{\lm}+\frac{1}{\eta+q-1-\ti}+
p\frac{N(\eta-1)+\om\eta}{\lm\bra{p\eta+(N-\om\eta)(p-2)}}
+
\frac{\eta(N-q)}{(\eta+q-1-\ti)\bra{p\eta+(N-\om\eta)(p-2)}}.
\end{align*}
This  means that, as $j\to\infty$, we need to have $\vp\ge 0$ in order to have \eqref{dis222} fulfilled. This is equivalent to the condition
\[
-(\eta+q-1-\ti)(N-\om\eta)+\eta(N+p-q-\om(p-2))+N(p-2)\ge0,
\]
which leads to
\[
\eta(p-q)-N(q-p+1-\ti)+\om\eta(q-p+1-\ti+\eta)\ge0
\]
from which we deduce
\begin{equation}\label{sieta}
\si-\eta\le \frac{\om\eta}{p-q}(q-p+1-\ti+\eta),
\end{equation}
thanks to the definition of $\si$.\\
Since one can choose $\om$ sufficiently small such that \eqref{sieta} is violated, then one gains a contradiction which shows that \eqref{solne1} can not be fulfilled.

\section{Comments on the superlinear thresholds and data}
\label{appendix}

We want to describe the connection between the growth assumptions on gradient term and the behavior at infinity of the function $g $. Hence, we look for the thresholds which discriminate the super/sublinear behavior of the gradient term with respect to the parameters $q$ and $\ti$. Moreover, we put in evidence the link between $q$, $\ti$ and the Lebesgue space where the initial datum $u_0$ has to be taken in order to have an existence result when the superlinear trend occurs. We remark that we have a regularizing effect on the superlinear threshold, in the sense that our threshold is greater than the one related to the case $\ti=0$. Roughly speaking, this means that we are "sublinear" for more values of $q$ with respect to the case $\ti=0$. Another regularizing effect concerns the initial datum $u_0$, which is allowed to belong to larger Lebesgue spaces with respect to the case $\ti=0$.\\
Finally, we also determine the minimal  regularity needed on source terms. In this case, we analyze the connection among the function $h $ at infinity and $f$ in the superlinear setting.\\
It is worth to point out that, also in  these cases, we are able to weaken the regularity on $f$, exploiting the degeneracy of $h $ at infinity.

\subsection{The superlinear threshold and the initial datum}\label{st}

Let us consider the following parabolic Cauchy-Dirichlet problem:
\begin{equation}\label{eqf}
\begin{cases}
\ds u_t-\D_p u=\frac{F}{(1+u)^\ti}  & \t{in }\,\,Q_T,\\
u\ge0&\t{in }\,\,Q_T,\\
 u=0  &\t{on }\,\,(0,T)\times\partial \Omega,\\
 u(0,x)=0  &\t{in } \,\,\Omega.
\end{cases}
\end{equation}
We suppose that the forcing term satisfies
\begin{equation*}
0\le F\in L^k(Q_T), 
\end{equation*}
$k$ to be fixed.\\

In order to highlight the roles of parameters involved, it is without loss of generality that we assume a zero initial datum $u_0$, and avoid the singular term keeping memory of the $\ti$ power as we do on the right-hand of \eqref{eqf}.

We briefly explain what we want to prove: our aim is obtaining an inequality of the type
\[
\||\N u|^\ell\|_{L^1(Q_t)}\le c\|F\|_{L^k(Q_t)}^\rho\q\forall t\le T,
\]
for some $\ell,\,\rho$, so that, letting $F=|\N u|^q$, this reads
\[
\||\N u|^\ell\|_{L^1(Q_t)}\le c\||\N u|^q\|_{L^k(Q_t)}^\rho=c\||\N u|^\ell\|_{L^{\frac{kq}{\ell}}(Q_t)}^{\frac{\rho q}{\ell}}\q\forall t\le T.
\]
First of all, we need to formally impose $kq\le \ell$ in order to close the estimate. Furthermore, being interested in a superlinear growth in the gradient term, we look for the superlinear threshold given by $\rho/\ell$, hence we ask for $\ell<\rho q$.\\

We formally multiply \eqref{eqf} by $\left( (1+u)^{\si-1}-1 \right)$, with $\si\in (1,2)$ to be fixed, and we set $\ds v=(1+u)^{\frac{\si+p-2}{p}}$. We consider $t\le T$. In what follows, we assume that $\ti<\si-1$ since, on the contrary, the problem would loose its superlinear trend. \\
Applying H\"older's inequality with indices $(k,k')$, we get
\begin{equation}\label{ds4}
\int_{\Omega} (v(t))^{\frac{p\si}{\si+p-2}} +\iint_{Q_t}|\N v|^p  \le c\|F\|_{L^{k}(Q_t)}\|v\|_{L^{k'\frac{p(\si-1-\ti)}{\si+p-2}}(Q_t)}^{\frac{p(\si-1-\ti)}{\si+p-2}}+c.
\end{equation}

If $\frac{2N}{N+\si}<p<N$, then $\frac{p\si}{\si+p-2}<p^*$ and hence we can apply Theorem \ref{teoGN} with $h=\frac{p\si}{\si+p-2}$, $\eta=p$ and $a=b$, gaining the regularity $v\in L^{p\frac{N+\frac{p\si}{\si+p-2}}{N}}(Q_t)$. 
We proceed requiring  $k'(\si-1-\ti)\frac{p}{\si+p-2}=p\frac{N+\frac{p\si}{\si+p-2}}{N}$, and thus the value of $\si$ is given by the following expression:
\begin{equation}\label{nua}
\si=\si(k)=N\frac{k(p+\ti-1)-(p-2)}{N-p(k-1)},\q\t{i.e.} \q k=k(\si)=\frac{N(\si+p-2)+p\si}{N(p+\ti-1)+p\si}.
\end{equation}
Note that the value of $k$ in \eqref{nua} is strictly greater than $1$, because we have assumed $\ti<\si-1$. 
Moreover, the inequality \eqref{disGN=} with \eqref{ds4} implies that
\begin{equation*}
\|v\|_{L^{p\frac{N+\frac{p\si}{\si+p-2}}{N}}(Q_t)}^{p\frac{N+\frac{p\si}{\si+p-2}}{N}}\le 
c\|F\|_{L^k(Q_t)}^{\frac{N+p}{N}}
\|v\|_{L^{p\frac{N+\frac{p\si}{\si+p-2}}{N}}(Q_t)}^{p\frac{N+\frac{p\si}{\si+p-2}}{N}\frac{N+p}{k'N}}+c,
\end{equation*}
and thus, being $\frac{N+p}{k'N}<1$ by the the definition on $k$ and the lower bound on $p$, we can apply Young's inequality, 
getting
\begin{equation}\label{ds2}
\|v\|_{L^{p\frac{N+\frac{p\si}{\si+p-2}}{N}}(Q_t)}^{p\frac{N+\frac{p\si}{\si+p-2}}{N}}\le c\|F\|_{L^k(Q_t)}^{\frac{k(N+p)}{N-p(k-1)}}+c.
\end{equation}

Now, let us focus on the gradient term: consider the integral   of $|\N u|^\ell$, $\ell$ to be chosen. Then we estimate as
\begin{equation}\label{ds3}
\begin{split}
\iint_{Q_t}|\N u|^\ell  &\le \left( \iint_{Q_t}\frac{|\N u|^p}{(1+u)^{2-\si}}   \right)^{\frac{\ell}{p}} \left(\iint_{Q_t} (1+u)^{\frac{\ell(2-\si)}{p-\ell}}  \right)^{\frac{p-\ell}{p}}\le c \left( \iint_{Q_t}|\N v|^p   \right)^{\frac{\ell}{p}} \left(\iint_{Q_t} v^{\frac{\ell(2-\si)}{p-\ell}\frac{p}{\si+p-2}}  \right)^{\frac{p-\ell}{p}}.
\end{split}
\end{equation}
Requiring $\ds \frac{\ell(2-\si)}{p-\ell}\frac{p}{\si+p-2} =p\frac{N+\frac{p\si}{\si+p-2}}{N} $, we find the following expressions of $\ell$:
\begin{equation}\label{bnu}
\ell=\ell(\si)=\frac{N(\si+p-2)+\si p}{N+\si},\q\t{i.e.}\q \ell=\ell(k)=k\frac{p(N+1+\ti)-N(1-\ti)}{N-k(1-\ti)+2}.
\end{equation}
In particular
\begin{equation}\label{kb}
k=k(\ell)=\ell\frac{N+2}{p(N+1+\ti)-(N-\ell)(1-\ti)}.
\end{equation}
We continue requiring that  $k<\left(p\frac{N+2}{N(1-\ti)}\right)'$, in order to have $\ell<p$.
\\
Thanks to \eqref{ds4}, we estimate \eqref{ds3} as
\begin{align*}
\| |\N u|^\ell\|_{L^1(Q_t)}&\le c\norm{F}_{L^k(Q_t)}^\frac{\ell}{p}\norm{v}_{L^{p\frac{N+\frac{p\si}{\si+p-2}}{N}}(Q_t)}^{\ell\frac{N+\frac{p\si}{\si+p-2}}{k'N} +(p-\ell)\frac{N+\frac{p\si}{\si+p-2}}{N}}+c
&\le c\norm{F}_{L^k(Q_t)}^{\frac{k(N+p-\ell)}{N-p(k-1)}}+c
&=c\|F\|_{L^k(Q_t)}^{\frac{\ell(N+2)}{p(N+1+\ti)-N(1-\ti)}}+c,
\end{align*}
thanks also to \eqref{ds2} and the values of $\si=\si(k)$ in \eqref{nua} and $k=k(\ell)$ in \eqref{kb}. We let $F=|\N u|^q$. Thus, our last estimate becomes
\begin{equation*}
\| |\N u|^\ell\|_{L^1(Q_t)}\le c\||\N u|^\ell\|_{L^{\frac{kq}{\ell}}(Q_t)}^{\frac{q}{\ell}\frac{\ell(N+2)}{p(N+1+\ti)-N(1-\ti)}}+c.
\end{equation*} 

The first thing we have to check is the relation between the exponents of the Lebesgue spaces  which allow us to close the estimate above.
Then, we ask  for
\[
kq\le \ell.
\]
We observe that the inequality $kq\le \ell$, \eqref{nua} and \eqref{bnu} determine where the initial datum has to be taken, i.e. $u_0\in L^\si(\Omega)$ where $\si=\si(p,q,\ti,N)$ satisfies
\begin{equation*} 
\si\ge \frac{N(q-p+1-\ti)}{p-q}.
\end{equation*}
In particular, the sharp Lebesgue space of the initial data is $L^\si(\Omega)$ where $\ds\si=\frac{N(q-p+1-\ti)}{p-q}$.\\
 
We conclude this case by requiring
\[
1<\frac{q}{\ell}\frac{\ell(N+2)}{p(N+1+\ti)-N(1-\ti)},
\]
which provides us with the superlinear threshold
\begin{equation*} 
q>\frac{N(p-1+\ti)+p(1+\ti)}{N+2}.
\end{equation*}

\medskip
 
If $1<p\le \frac{2N}{N+\si}$, then $\frac{p\si}{\si+p-2}\ge p^*$ and thus we cannot apply Theorem \ref{teoGN}, as done in \eqref{ds2}. \\
However, we know that, at least, we can impose $\ds k'(\si-1-\ti)=\si$ in \eqref{ds4} (i.e. $ k={\si}/{(1+\ti)}$). Inequality \eqref{ds4} becomes
\begin{equation}\label{bla}
\int_{\Omega} (v(t))^{\frac{p\si}{\si+p-2}} +\iint_{Q_t}|\N v|^p  \le c\|F\|_{L^{k}(Q_t)}\|v\|_{L^{\frac{p\si}{\si+p-2}}(Q_t)}^{\frac{1}{k'}\frac{p\si}{\si+p-2}}+c.
\end{equation}
We then apply Young's inequality to the right-hand above with $(k,k')$, and recover
\begin{equation}\label{f}
\|v\|_{L^{\frac{p\si}{\si+p-2}}(Q_t)}^{\frac{p\si}{\si+p-2}}\le c\|F\|_{L^{k}(Q_t)}^k+c.
\end{equation}
We set $\ds \frac{\ell(2-\si)}{p-\ell}={\si}$ in \eqref{ds3} (i.e. $ \ell={\si p}/{2}$), from which
\begin{align*}
\| |\N u|^\ell\|_{L^1(Q_t)}&\le
c\norm{F}_{L^k(Q_t)}^\frac{\ell}{p}\norm{v}_{L^{\frac{p\si}{\si+p-2}}(Q_t)}^{\frac{\ell}{k'}\frac{\si}{\si+p-2}+\frac{p\si}{\si+p-2}\frac{p-\ell}{p}}+c
\le c\norm{F}_{L^k(Q_t)}^k+c,
\end{align*}
thanks also to \eqref{bla} and \eqref{f}. 
The values of $\ell,\,k$ imply that 
\[
\| |\N u|^\ell\|_{L^1(Q_t)}=\| |\N u|^\frac{\si p}{2}\|_{L^1(Q_t)}\le c\|F\|_{L^{\frac{\si}{1+\ti}}(Q_t)}^{\frac{\si}{1+\ti}}+c
\]
which, for $F=|\N u|^q$, becomes
\[
\| |\N u|^\frac{\si p}{2}\|_{L^1(Q_t)}\le c\| |\N u|^q\|_{L^{\frac{\si}{1+\ti}}(Q_t)}^{\frac{\si}{1+\ti}}+c=
c\| |\N u|^\frac{\si p}{2}\|_{L^{\frac{2q}{p(1+\ti)}}(Q_t)}^{\frac{2q}{p(1+\ti)}}+c.
\]
A suitable Lebesgue spaces inclusion occurs whenever 
\[
q\le\frac{p(1+\ti)}{2}.
\]
It is important to underline that this estimate holds for every value of $p$.\\

It is worth to point out that,  if $\si=\frac{N(q-p+1-\ti)}{p-q}$, then it holds that
\[
p>\frac{2N}{N+\si}\q\Longleftrightarrow\q q>\frac{p(1+\ti)}{2}.
\]
This means that, as $p\le \frac{2N}{N+\si}$, we fall within the case of sublinear gradient growth.
Thus, the superlinear threshold is given by
\begin{equation*} 
\max\set{\frac{p(1+\ti)}{2},\frac{N(p-1+\ti)+p(1+\ti)}{N+2}}=\begin{cases}
\ds\frac{p(1+\ti)}{2}&\t{if}\q 1<p<2;\\
\ds\frac{N(p-1+\ti)+p(1+\ti)}{N+2}&\t{if}\q p\ge 2.
\end{cases}
\end{equation*}

\subsection{On the forcing term}

We now wonder which is the regularity on the forcing term we have to ask for  when $p>\frac{2N}{N+\si}$. Then, we let $0\le F\in L^{r}(0,T;L^{m}(\Omega))$, and look for the curve where the exponents $(m,r)$ need to be taken in order to have an existence result. We thus consider 
\begin{equation*} 
\begin{cases}
\ds u_t-\D_p u=\frac{F}{(1+u)^\gi}  & \t{in }\,\,Q_T,\\
 u=0  &\t{on }\,\,(0,T)\times\partial \Omega,\\ u(0,x)=0  &\t{in } \,\,\Omega,
\end{cases}
\end{equation*}
Let  $\gi<\si-1$. 
We recall the inequality in \eqref{ds4} where $\ds v=(1+u)^{\frac{\si+p-2}{p}}$ and, thanks to two applications of H\"older's inequality with exponents $(m,m')$ and $(r,r')$ on the right-hand, we get
\begin{equation*} 
\begin{split}
\int_{\Omega} (v(t))^{\frac{p\si}{\si+p-2}} +\iint_{Q_t}|\N v|^p  &\le c\|F\|_{L^{r}(0,T;L^{m}(\Omega))}\|v\|_{L^{\overline{r}}(0,T;L^{\overline{m}}(\Omega))}^{\frac{p(\si-1)}{\si+p-2}}+c
\end{split}
\end{equation*}
where
\[
\overline{r}=r'(\si-1-\gi)\frac{p}{\si+p-2}\qq\t{and}\qq \overline{m}=m'(\si-1-\gi)\frac{p}{\si+p-2}.
\]
Note that here we need to have $\gi<\si-1$.\\
We apply again Theorem \ref{teoGN} for $h=\frac{p\si}{\si+p-2}$ and $\eta=p$ but we now focus on the case $a\ne b$. Then, we have
\begin{equation*} 
\int_0^t \|v(s)\|_{L^{b}(\Omega)}^{a} \le c\|v\|_{L^\infty(0,t;L^{\frac{p\si}{\si+p-2}}(\Omega))}^{a-p}\int_0^t\|\N v(s)\|_{L^p(\Omega)}^p ,
\end{equation*}
where $(a,b)$ satisfy the relation
\begin{equation}\label{gn}
\frac{N\frac{p\si}{\si+p-2}}{b}+\frac{N(p-\frac{p\si}{\si+p-2})+p\frac{p\si}{\si+p-2}}{a}=N.
\end{equation}
Observe that, if $r\ne m$, then $a\ne b$ and vice versa.\\
We proceed requiring $b\ge\overline{m}$ and $a\ge\overline{r}$ which, by \eqref{gn}, gives us
\begin{equation*}
\frac{N\si}{m}+\frac{N(p-2)+p\si}{r}\le N(p-1+\gi)+ p\si
\end{equation*}
otherwise, we have found the curve of the values $(m,r)$ in term of the summability of the initial datum $u_0\in L^\si(\Omega)$.

Finally, if $\gi\ge\si-1$,  the power $(1+u)^\gi$ is strong enough in order to weaken the regularity assumptions on $f$ (not on $u_0$) and the estimate is closed with just $L^1(Q_T)$ source data.

\subsection{Conclusion}

We sketch the thresholds of $q$ just found below.
\begin{center}
\begin{tabular}{rl}
\begin{tikzpicture}
\draw [very thick, dashed,color=red!70!black] (-1,0) to (2,0);
\end{tikzpicture}  
&   $\si\ge2$
\\
\begin{tikzpicture}
\draw [very thick, dotted,color=orange] (-1,0) to (2,0);
\end{tikzpicture}  
&   $1<\si<2$
\\
\begin{tikzpicture}
\draw [very thick, color=yellow] (-1,0) to (2,0);
\end{tikzpicture}
& $\si<1$
\end{tabular}
\end{center}
We set
\begin{align*}
q_1=\frac{N(p-1+\ti)+p(1+\ti)}{N+2},\qq
q_2=\frac{p(1+\ti)}{2}
.
\end{align*}

The red zone below regards values of $q$ belonging to the interval $\bigl[p-\frac{N(1-\ti)}{N+2},p\bigr)$ which implies that $\si\ge 2$: this means that we expect to have solutions with finite energy. In the orange zone, we have that  $\si\in (1,2)$. Finally, the last yellow range implies that $0<\si\le 1$,  it cannot be considered as Lebesgue exponent. The superlinear threshold is given by $q_1$.
\begin{figure}[H]
\centering
\begin{tikzpicture}
\draw [thin] (0,0) to (3.25,0);
\draw [->,thin] (10,0) -- (11,0);
\draw [very thick,  dashed,color=red!70!black] (8,0) to (10,0);
\draw [very thick, dotted,color=orange] (6,0) to (8,0);
\draw [very thick, color=yellow] (3.25,0) to (6,0);
\fill (0,0) circle (2pt) node[below] 
{$0$};
\fill (1,0) circle (2pt) node[below] {$q_2$};

\draw (3.25,0) node[circle=2pt,
       rounded corners,
		draw=yellow,
		fill=white,
       thick,
		inner sep=1.5pt,minimum size=1pt,
       ](){};
\fill (3.25,0) circle (0pt) node[below]
{$q_1$};

\draw (6,0) node[circle=2pt,
       rounded corners,
		draw=yellow,
		fill=yellow,
       thick,
		inner sep=1.5pt,minimum size=1pt,
       ](){};
\fill (6,0) circle (0pt) node[below]
{$p-\frac{N(1-\ti)}{N+1}$};

\draw (8,0) node[circle=2pt,
       rounded corners,
		draw=red!70!black,
		fill=red!70!black,
       thick,
		inner sep=1.5pt,minimum size=1pt,
       ](){};
\fill (8,0) circle (0pt) node[below]
{$p-\frac{N(1-\ti)}{N+2}$};

\draw (10,0) node[circle=2pt,
       rounded corners,
		draw=red!70!black,
		fill=white,
       thick,
		inner sep=1.5pt,minimum size=1pt,
       ](){};
\fill (10,0) circle (0pt) node[below]
{$p$};

\end{tikzpicture}
\caption{The case $p\ge 2$}
\label{fig:1}
\end{figure}
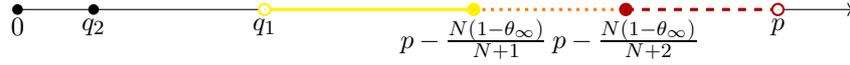

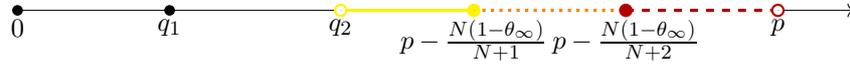
\begin{figure}[H]
\centering
\begin{tikzpicture}
\draw [thin] (0,0) to (4.25,0);
\draw [->,thin] (10,0) -- (11,0);
\draw [very thick,  dashed,color=red!70!black] (8,0) to (10,0);
\draw [very thick, dotted,color=orange] (6,0) to (8,0);
\draw [very thick, color=yellow] (4.25,0) to (6,0);
\fill (0,0) circle (2pt) node[below] 
{$0$};

\fill (2,0) circle (2pt) node[below] {$q_1$};

\draw (4.25,0) node[circle=2pt,
       rounded corners,
		draw=yellow,
		fill=white,
       thick,
		inner sep=1.5pt,minimum size=1pt,
       ](){};
\fill (4.25,0) circle (0pt) node[below]
{$q_2$};

\draw (6,0) node[circle=2pt,
       rounded corners,
		draw=yellow,
		fill=yellow,
       thick,
		inner sep=1.5pt,minimum size=1pt,
       ](){};
\fill (6,0) circle (0pt) node[below]
{$p-\frac{N(1-\ti)}{N+1}$};

\draw (8,0) node[circle=2pt,
       rounded corners,
		draw=red!70!black,
		fill=red!70!black,
       thick,
		inner sep=1.5pt,minimum size=1pt,
       ](){};
\fill (8,0) circle (0pt) node[below]
{$p-\frac{N(1-\ti)}{N+2}$};

\draw (10,0) node[circle=2pt,
       rounded corners,
		draw=red!70!black,
		fill=white,
       thick,
		inner sep=1.5pt,minimum size=1pt,
       ](){};
\fill (10,0) circle (0pt) node[below]
{$p$};
\end{tikzpicture}
\caption{The case $\frac{2N}{N+1}<p< 2$}
\label{fig:2}
\end{figure}
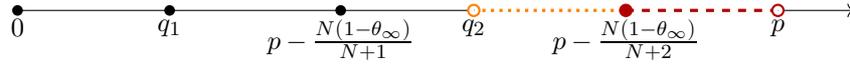
\begin{figure}[H]
\centering
\begin{tikzpicture}
\draw [thin] (0,0) to (6,0);
\draw [->,thin] (10,0) -- (11,0);
\draw [very thick,  dashed,color=red!70!black] (8,0) to (10,0);
\draw [very thick, dotted,color=orange] (6,0) to (8,0);
\fill (0,0) circle (2pt) node[below] 
{$0$};

\fill (2,0) circle (2pt) node[below] {$q_1$};
\fill (4.25,0) circle (2pt) node[below] {$p-\frac{N(1-\ti)}{N+1}$};

\draw (6,0) node[circle=2pt,
       rounded corners,
		draw=orange,
		fill=white,
       thick,
		inner sep=1.5pt,minimum size=1pt,
       ](){};
\fill (6,0) circle (0pt) node[below]
{$q_2$};

\draw (8,0) node[circle=2pt,
       rounded corners,
		draw=red!70!black,
		fill=red!70!black,
       thick,
		inner sep=1.5pt,minimum size=1pt,
       ](){};
\fill (8,0) circle (0pt) node[below]
{$p-\frac{N(1-\ti)}{N+2}$};

\draw (10,0) node[circle=2pt,
       rounded corners,
		draw=red!70!black,
		fill=white,
       thick,
		inner sep=1.5pt,minimum size=1pt,
       ](){};
\fill (10,0) circle (0pt) node[below]
{$p$};
\end{tikzpicture}
\caption{The case $\frac{2N}{N+2}<p\le \frac{2N}{N+1}$}
\label{fig:3}
\end{figure}
\begin{figure}[H]
\centering
\begin{tikzpicture}
\draw [thin] (0,0) to (8,0);
\draw [->,thin] (10,0) -- (11,0);
\draw [very thick,  dashed,color=red!70!black] (8,0) to (10,0);
\fill (0,0) circle (2pt) node[below] 
{$0$};
\fill (8,0) circle (2pt) node[below] {$q_2$};
\fill (2,0) circle (2pt) node[below] {$q_1$};
\fill (4.25,0) circle (2pt) node[below] {$p-\frac{N(1-\ti)}{N+1}$};
\fill (6.4,0) circle (2pt) node[below] {$p-\frac{N(1-\ti)}{N+2}$};
\draw (8,0) node[circle=2pt,
       rounded corners,
		draw=red!70!black,
		fill=white,
       thick,
		inner sep=1.5pt,minimum size=1pt,
       ](){};
\fill (8,0) circle (0pt) node[below]
{$q_2$};

\draw (10,0) node[circle=2pt,
       rounded corners,
		draw=red!70!black,
		fill=white,
       thick,
		inner sep=1.5pt,minimum size=1pt,
       ](){};
\fill (10,0) circle (0pt) node[below]
{$p$};
\end{tikzpicture}
\caption{The case $\frac{2N}{N+\si}<p\le \frac{2N}{N+2}$}
\label{fig:4}
\end{figure}
Figures \ref{fig:2} and \ref{fig:3} show how the superlinear threshold $q_2$ and the ''energy'' data threshold $p-\frac{N(1-\ti)}{N+2}$ are placed.\\

We conclude highlighting the gain of regularity  with respect to the case $\ti=\gi=0$ as far as the source term $f$ is concerned. Here, $\si_0=\si\bigr|_{\ti=0}$,
\begin{align*}
x_{\ti,\gi}=\frac{p-1+\gi}{\si}+\frac{p}{N},\q y_{\ti,\gi}=\pare{\frac{p-1+\gi}{\si}+\frac{p}{N}}\frac{N\si}{N(p-2)+p\si}, \\
x_0=x_{\ti,\gi}\bigr|_{\ti=\gi=0}=\frac{p-1}{\si_0}+\frac{p}{N},\q y_0=y_{\ti,\gi}\bigr|_{\ti=\gi=0}=\pare{\frac{p-1}{\si_0}+\frac{p}{N}}\frac{N\si_0}{N(p-2)+p\si_0}
\end{align*}
\begin{tabular}{cc}
\begin{minipage}{.4\textwidth}

\begin{figure}[H]
\centering
\scalebox{1}{\begin{tikzpicture} 
 
\fill [domain=0:3, color=red!70!black, 
					variable=\x](0,0)
					-- plot ({\x}, {3-\x})
					-- (3,0)
					-- cycle;
\draw[domain=0:2,smooth,very thick,variable=\x,red!70!black] (0,0) plot ({\x}, {2-\x});

\fill [domain=0:2.5, color=blue!70, pattern=north east lines, pattern color=blue!70, 
					variable=\x](0,0)
					-- plot ({\x}, {2.5-\x})
					-- (2.5,0)
					-- cycle;
\draw[domain=0:2.5,smooth,very thick,variable=\x,red!70!black] (0,0) plot ({\x}, {2.5-\x});

\draw[->] (0,0) -- (3.5,0) ;
\draw  (3.7,.5) node[anchor=north west] {$\frac{1}{m}$};
\draw[->] (0,0) -- (0,3.8) ;
\draw   (0,4) node[anchor=north east] {$\frac{1}{r}$};  

\fill (0,3) circle (1.5pt);
\fill (0,3.2) circle (0pt) node[anchor=north east] {$y_{\ti,\gi}$};

\fill (3,0) circle (1.5pt);
\fill (3.2,0) circle (0pt) node[below] {$x_{\ti,\gi}$};

\fill (0,2.5) circle (1.5pt) node[anchor=north east] {$y_0$};

\fill (2.5,0) circle (1.5pt);
\fill (2.3,0) circle (0pt) node[below] {$x_0$};
\end{tikzpicture}
}
\end{figure}
\end{minipage}
&
\begin{minipage}[c]{.6\textwidth}

\begin{tabular}{cc}
\begin{minipage}[c]{.1\textwidth}
\hspace*{-.5cm}
\begin{tikzpicture}[scale=.5]
\fill [domain=0:1, color=red!70!black, 
					variable=\x](0,0)
					-- plot ({\x}, {1-\x})
					-- (1,0)
					-- cycle;
\draw[domain=0:1,smooth,very thick,variable=\x,red!70!black] (0,0) plot ({\x}, {1-\x});
\end{tikzpicture}
\end{minipage}
&\hspace*{-1cm}
\begin{minipage}[c]{.9\textwidth}
admissible values of $\pare{\frac{1}{m},\frac{1}{r}}$ for $\ti,\,\gi\ne 0$
\end{minipage}
\\
\vspace*{.2cm}
&
\\
\begin{minipage}[c]{.1\textwidth}
\hspace*{-.5cm}
\begin{tikzpicture}[scale=.5]
\fill [domain=0:1, color=blue!70, pattern=north east lines, pattern color=blue!70, 
					variable=\x](0,0)
					-- plot ({\x}, {1-\x})
					-- (1,0)
					-- cycle;

\end{tikzpicture}
\end{minipage}
&\hspace*{-1cm}
\begin{minipage}[c]{.9\textwidth}
admissible values of $\pare{\frac{1}{m},\frac{1}{r}}$ for $\ti,\,\gi= 0$
\end{minipage}
\end{tabular}

\end{minipage}
\end{tabular}

\end{document}